\documentclass[11pt]{amsart}

\usepackage{amsmath,amsfonts,amsthm,amssymb,amscd}
\usepackage{epsfig}
\usepackage{inputenc}
\usepackage{upref}
\usepackage{bm}
\usepackage{calc}
\usepackage{caption}
\usepackage{subcaption}
\usepackage{hyperref}
\usepackage[usenames,dvipsnames]{xcolor}
\usepackage[normalem]{ulem}
\usepackage{wasysym}
\usepackage{mathrsfs}
\usepackage{enumerate}
\usepackage{setspace}
\usepackage{array}
\usepackage{cancel}
\usepackage{tikz-cd}
\usepackage[overload]{empheq}
\usepackage[title]{appendix}

\usepackage{colonequals}
\usepackage{microtype}

\numberwithin{equation}{section}

\newtheorem{theorem}{Theorem}[section]
\newtheorem*{theorem*}{Theorem}

\newtheorem{lemma}[theorem]{Lemma}
\newtheorem{proposition}[theorem]{Proposition}

\newtheorem{definition}[theorem]{Definition}

\newtheorem{question}[theorem]{Question}

\DeclareMathOperator{\map}{Map}
\DeclareMathOperator{\Conf}{Conf}

\DeclareMathOperator{\End}{End}

\DeclareMathOperator{\Sym}{Sym}
\DeclareMathOperator{\Area}{Area}

\newcommand{\R}{\mathbb{R}}

\newcommand{\C}{\mathbb{C}}

\newcommand{\Z}{\mathbb{Z}}

\newcommand{\s}{\mathfrak{s}}
\newcommand{\F}{\mathfrak{F}}

\newcommand{\spinc}{$spin^c\ $}
\newcommand{\half}{\frac{1}{2}}
\newcommand{\E}{\mathcal{E}}

\newcommand{\re}{\text{Re}}

\newcommand{\SO}{\mathcal{O}}

\newcommand{\SC}{\mathcal{C}}
\newcommand{\SB}{\mathcal{B}}

\newcommand{\SG}{\mathscr{G}}

\newcommand{\Remark}{\textit{Remark.}\ }

\newcommand{\CP}{\mathbb{CP}^1}

\newcommand{\Id}{\text{Id}_{S}}

\newcommand{\Hom}{\text{Hom}}

\newcommand{\A}{\mathcal{A}}
\newcommand{\supp}{\text{supp}}

\newcommand{\SL}{\mathscr{L}}
\newcommand{\CL}{\mathcal{L}}
\newcommand{\embed}{\hookrightarrow}
\newcommand{\T}{\mathbb{T}}
\newcommand{\bpartial}{\bar{\partial}}

\newcommand{\Vol}{\text{Vol}}

\newcommand{\talpha}{{\tilde{\alpha}}}
\newcommand{\SH}{\mathcal{H}}

\newcommand{\Du}{\frac{D}{du}}
\newcommand{\Dv}{\frac{D}{dv}}

\title{On finite energy monopoles on $\C\times \Sigma$}
\author{Donghao Wang}

\date{\today}

\address{Department of Mathematics, Massachusetts Institute of Technology, Cambridge, MA 02139, USA}
\email{donghaow@mit.edu}

\begin{document}
	
\begin{abstract} We classify solutions to the Seiberg-Witten equations on $X=\C\times \Sigma$ with finite analytic energy. The spin bundle $S^+\to X$ splits as $L^+\oplus L^-$. When $2-2g\leq c_1(S^+)[\Sigma]<0$, the moduli space is identified with the moduli space of pairs $((L^+,\bpartial), f)$ where $(L^+,\bpartial)$ is a holomorphic structure on $L^+$ and $f: \C\to H^0(\Sigma, L^+,\bpartial)$ is a polynomial map. Moreover, the solution has analytic energy $-4\pi^2d\cdot c_1(S^+)[\Sigma]$ if $f$ has degree $d$. 

When $c_1(S^+)=0$, all solutions are reducible and it is the space of flat connections on $\bigwedge^2 S^+$.  

 Solutions will have either exponential decay or power law decay according to the polynomial map $f$. We give a complete criterion for different cases. 
\end{abstract}

\maketitle
\tableofcontents

\section{Introduction}

\subsection{Motivation in Floer Homology}The purpose of this paper is to give a complete classification of finite energy monopoles on $X=\C\times\Sigma$. This classification problem arises naturally in the context of Floer theory of 3-manifolds with cylindrical ends.

The Seiberg-Witten Floer Homology is defined for arbitrary closed oriented 3-manifold $Y$ by Kronheimer-Mrowka in \cite{Bible} and has greatly influenced the study of 3-dimensional topology. The underlying idea is to construct infinite dimensional Morse theory: solutions to the 3-dimensional Seiberg-Witten equation on $Y$ are critical points of the Chern-Simons-Dirac functional $\CL$, and solutions to the 4-dimensional equation on $\R\times Y$ are viewed as negative gradient flowlines of $\CL$. We take the chain group to be the free abelian group generated by critical points of $\CL$. Differentials are given by counting flowlines that connect critical points with adjacent indices. In order to make this picture work, suitable perturbations of $\CL$ are needed. 

One reason to develop a relative version of Floer theory for 3-manifolds with boundaries is to give a gluing formula for the absolute version, which may facilitate computations in some cases.  This version may also give topological applications in its own right. This goal is partly accomplished in Heegaard Floer Homology, developed by Ozsv\'{a}th and Szab\'{o} in \cite{HF} as a symplectic geometric replacement for gauge theory. Their construction relies on Gromov's theory of pseudo-holomorphic curves. Some generalizations for 3-manifolds with boundaries include Knot Floer Homology \cite{KFH, KFH1} and Bordered Floer Homology \cite{BHF}. It is now known that Heegaard Floer Homology and Seiberg-Witten Floer Homology are equivalent \cite{CGH}\cite{HF=HMI}. However, the gauge theoretic counterparts of Knot Floer Homology and Bordered Floer Homology are still missing. 

Some attempts that avoid analytic technicalities have been made towards this direction. In \cite{KS}, the Seiberg-Witten Floer Homology was developed for balanced sutured 3-manifolds and a version of Knot Floer Homology was defined. On the other hand, Nyugen \cite{NguyenI, NguyenII} studied the monopole equation on $Y$ directly and developed analytic foundations for constructing Floer theories with the Lagrangian boundary condition on $\Sigma$. 

\bigskip
We shall now describe a more direct approach to this problem. Suppose we wish to define Floer-theoretic invariants for a compact oriented 3-manifold $Y$ with boundary $\Sigma$. We allow $\Sigma$ to have multiple connected components $(\Sigma_1,\cdots,\Sigma_m)$. We attach cylindrical ends to $Y$ and study the monopole equation on $Y^*=Y\coprod_\Sigma \R^{\geq 0}\times \Sigma$.  In this case, the moduli space of finite energy solutions on $Y$ is automatically compact and in general has positive formal dimensions. It is also known that each solution will converge to a vortex on $\Sigma_i$ as it approaches infinity along each boundary end.

So far we do not have any means to produce invariants of $Y$ out of this picture.  Suppose we go one step further and consider the moduli space of finite energy monopoles on $\R\times Y^*$, which is expected to produce differentials and plays a role in the definition of Floer theory. We would hope this moduli space has a nice compactification. However, for a sequence of solutions on $\R\times Y^*$, it is possible that some amount of energy escapes through the cylindrical ends of $Y^*$, which makes the moduli space non-compact. It is believed that finite energy monopoles on $X=\C\times\Sigma$ should serve as models for these escaping ``bubbles" and contribute to correction terms in the definition of differentials. The purpose of this paper is then to give a complete classification of these monopoles on $X=\C\times\Sigma$. 
\medskip

\subsection{Statement of Main Results}

Let $X=\C\times\Sigma$ be the product of the complex plane $\C$ and $\Sigma$, endowed with the product metric. On the complex plane $\C$, it is the standard Euclidean metric and $\Sigma$ is any compact Riemann surface with a Hermitian metric. Let $g=g(\Sigma)$ be the genus of $\Sigma$.  
The main result of this paper establishes a bijection between the moduli space of finite energy monopoles and an object that is algebraic in nature. 
\begin{theorem}\label{main}
	When $2-2g\leq c_1(S^+)[\Sigma]<0$, 
	there is a bijection between sets:
	\begin{multline*}
	\{\text{solutions to the Seiberg-Witten equation $(\ref{SWEQ})$ of finite energy}\}/\SG \leftrightarrow \\ \{(\bpartial_B, f): f\neq 0:\C\to H^0(\Sigma,L^+,\bpartial_B)\text{ is a polynomial map}\}/\SG_\C(\Sigma). 
	\end{multline*}
	Furthermore, for the finite energy monopole  $(A,\Phi)$ that corresponds to $(\bpartial_B,f)$, its analytic energy $\E_{an}(A,\Phi)$ equals $-4\pi^2d\cdot c_1(S^+)[\Sigma]$ and the zero locus of the spin section $Z(\Phi^+)$ agrees with $Z(f)$. Here, $d=\deg(f)$ is the degree of $f$. 
\end{theorem}
Fixing the degree $d$ of $f$, the object on the right corresponds to the space of divisors of the line bundle  $\pi_1^*\SO(d)\otimes \pi_2^*\SL\to \CP\times\Sigma$ that are nonzero at the fiber at infinity $\{\infty\}\times \Sigma$, as $\SL$ varying for all holomorphic structures on $L^+\to\Sigma$. If in addition $f\neq 0$ for any $z\in\C$, this is the space of holomorphic maps of degree $d$ from $\CP$ to $\Sym^m \Sigma$ where $m=c_1(L^+)[\Sigma]\geq 0$.
\smallskip

To clarify our notations, recall that a \spinc structure $\s$ on $X$ is a pair $(S,\rho)$ where $S=S^+\oplus S^-$ is the spin bundle, and the bundle map $\rho: T^*X\to \Hom(S,S)$ defines the Clifford multiplication. An element $(A,\Phi)$ in the configuration space $\SC(X,\s)$ consists of a smooth \spinc connection $A$ and a smooth section $\Phi$ of $S^+$. Let $A^t$ be the induced connection on $\bigwedge^2 S^+$ and $F^+_{A^t}$ be the self-dual part of the curvature form $F_{A^t}$. The Seiberg-Witten equation is defined on $\SC(X,\s)$ by the formula:
	\begin{equation}\label{SWEQ}
\left\{
\begin{array}{r}
\half \rho(F_{A^t}^+)-(\Phi\Phi^*)_0=0,\\
D_A^+\Phi=0,
\end{array}
\right.
\end{equation}
where $D_A^+$ is the Dirac operator and $(\Phi\Phi^*)_0$ is the traceless part of $\Phi\Phi^*$ as a bundle map $S^+\to S^+$. This equation is also called the monopole equation and solutions are called monopoles. We write $\F(A,\Phi)$ for the formulae on the right and $(\ref{SWEQ})$ is equivalent to that $\F(A,\Phi)=0$.

The gauge group $\SG(X)=\map (X,S^1)$ acts naturally on $\SC(X,\s)$:
\[
\SG(X)\ni u: \SC(X,\s)\to \SC(X,\s),\ (A,\Phi)\mapsto (A-u^{-1}du, u\Phi).
\]
The monopole equation (\ref{SWEQ}) is invariant under gauge transformations. 
 
We are interested in the space of solutions to (\ref{SWEQ}) modulo  gauge assuming finiteness of the analytic energy: 
\begin{equation}\label{an}
\E_{an}(A,\Phi)=\int_X \frac{1}{4}|F_{A^t}|^2+|\nabla_A\Phi|^2+\frac{1}{4}|\Phi|^4+ \frac{K}{2}|\Phi|^2,
\end{equation}
where $K$ is the Gaussian curvature of $\Sigma$. This is the main object that appears on the left hand side of the bijection in Theorem \ref{main}. 

To justify the choice of $\E_{an}$, recall that for a closed 4-manifold $X$, $(A,\Phi)$ solves the monopole equation (\ref{SWEQ}) if and only if it minimizes the analytic energy; indeed, we have the energy formula 
\[
\E_{an}(A,\Phi)-\E_{top}=\int_X |\F(A,\Phi)|^2
\]
where the topological energy $\E_{top}$ depends only on characteristic classes of $S^+$. A similar energy formula in the context of the non-compact manifold $X=\C\times\Sigma$ is proved in Lemma \ref{53}, where the topological energy $\E_{top}=-4\pi^2 d\cdot c_1(S^+)[\Sigma]$ and $d$ is an integer. Therefore, a monopole on $X$ is not necessarily a global minimizer of the analytic energy, but it does minimize $\E_{an}$ in a suitable smaller variational space. 
\smallskip

To explain the second object in Theorem \ref{main}, let $dvol_\C$ and $dvol_\Sigma$ denote volume forms on $\C$ and $\Sigma$ respectively. Since the symplectic form $\omega=dvol_\C+dvol_\Sigma$ on $X$ is parallel, the spin bundle $S^+$ splits as $L^+\oplus L^-$: they are $\mp 2i$ eigenspace of $\rho(\omega)$. The spin section $\Phi$ then decomposes as $(\Phi_+,\Phi_-)$ with $\Phi_\pm\in \Gamma(X,L^\pm)$. The first observation is that finite energy monopoles are in fact vortices on $X$:

\begin{theorem}\label{coro} 
If there exists a smooth solution $(A,\Phi)$ to the monopole equation $(\ref{SWEQ})$ on $X$ with $\Phi\not\equiv 0$ and $\E_{an}(A,\Phi)<\infty$, then $0<|c_1(S^+)|\leq 2g-2$.  In addition, if $c_1(S^+)> (resp.< )\ 0$, then $\Phi_+\  (resp.\ \Phi_- )\equiv 0$. 
\end{theorem}

 Here, $c_1(S^+)$ is the Chern class associated to $S^+$. The same symbol is used to denote the pairing $c_1(S^+)[\Sigma]\in \Z$. Reducible solutions occur only if $c_1(S^+)=0$. The converse is also true:

\begin{theorem}\label{flat}
If $c_1(S^+)=0$, then $\Phi\equiv 0$ and the induced connection $A^t$ on $\bigwedge^2S^+$ is flat. 
\end{theorem} 

Replacing the complex structure on $X$ by its complex conjugate will interchange the bundles $L^+$ and $L^-$. We focus on the case when $2-2g\leq c_1(S^+)<0$ and $\Phi_-\equiv 0$. Choose a holomorphic structure $\bpartial_B$ on $L^+$ and let $\nabla_B$ be the Chern connection associated to $\bpartial_B$ .  We say a map $f:\C\to H^0(\Sigma,L^+,\bpartial_B)$ is a polynomial map of degree $d$ if $f$ is a polynomial function on $\C$ with coefficients in $H^0(\Sigma,L^+,\bpartial_B)$. That is to say, we can find $\gamma_i\in H^0(\Sigma, L^+,\bpartial_B)$ for $0\leq i\leq d$ such that for any $z\in\C$,
\[
f(z)=\sum_{i=0}^d \gamma_iz^i.
\]
The group of complex gauge transformations $\SG_\C(\Sigma)=\map(\Sigma, \C^*)$ acts on the pair $(\nabla_{B_0}, f)$ by the formula:
\[
\SG_\C(\Sigma)\ni u=u_1\cdot e^\alpha: (\nabla_{B_0}, f) \mapsto (\nabla_{B_0}-u_1^{-1}du_1+i*d\alpha, u\cdot f). 
\]
where $\alpha$ is real and $u_1\in \map(\Sigma,S^1)$. Therefore, the second object in Theorem \ref{main} is the quotient space of pairs $(\bpartial_B, f)$ by complex gauge transformations.

\medskip
We also complexify the gauge group $\SG(X)$ and define its action on $\SC(X,\s)$ by the same formula of $\SG_\C(\Sigma)$. Then $\SG_\C(X)=\SG(X)\times \Conf(X)$ where $\Conf(X)=\map(X,\R_+)$ corresponds to conformal transformations on $S^+$. What is hidden behind the correspondence in Theorem \ref{main} is that for any pair $(\nabla_B,f)$, we can find a conformal transformation $e^\alpha$ such that
\[
 (A,\Phi^+)=e^\alpha\cdot (\nabla_B+\frac{\partial}{\partial u}+\frac{\partial}{\partial v}, f)
\]
is a finite energy monopole. This is also true in the opposite direction. For precise statements, see Theorem \ref{polynomial} and Theorem \ref{existence}.
\bigskip

To analyze the dynamics of $(A,\Phi)$ at infinity, define the configuration space $\SC(\Sigma, L^+)$ in the same manner of $\SC(X,\s)$. Then any solution $(B,\sigma)\in \SC(\Sigma, L^+)$ to the vortex equation on $\Sigma$
\begin{equation}\label{vortex}
\left\{
\begin{array}{r}
i*F_B+\half K+\half |\sigma|^2=0\\
\bpartial_B\sigma=0
\end{array}
\right.
\end{equation}
gives a solution to $(\ref{SWEQ})$ on $X$. Indeed, one can pull back $(B,\sigma)$ over $\C$. This corresponds to the case when $\deg(f)=0$ and $\E_{an}(A,\Phi)=0$. 

It is convenient to introduce the quotient space $\SB(\Sigma, L^+)=\SC(\Sigma, L^+)/\SG(\Sigma)$. For each $k\geq 2$ define a metric on $\SB(\Sigma,L^+)$ by the formula
\[
d_k([a],[b])\colonequals \min_{u\in \SG(\Sigma)} \|u\cdot a-b\|_{L^2_k(\Sigma)},
\]
where $[a]$ and $[b]$ denote equivalent classes of $a,b\in \SC(\Sigma, L^+)$. 

When $\deg(f)=d>0$, by Bradlow's theorem \cite{Bradlow}, there is a solution $(B,\sigma)$ to (\ref{vortex}) such that $Z(\sigma)=Z(\gamma_d)$ where $\gamma_d$ is the leading coefficient of $f$. We know that  $d_k((A,\Phi^+)|_{\{z\}\times \Sigma}, (B,\sigma))\to 0$ as $z\to\infty$. The question is what is the decay rate. Suppose
\[
f=\gamma_d(z^d+a_{d-1}z^{d-1}+\cdots+a_{d-m+1}z^{d-m+1})+\gamma_{d-m}z^{d-m}+\cdots
\]
where $a_{d-i}\in \C$ are complex numbers and $\gamma_{d-m}$ is the highest coefficient that is not proportional to $\gamma_d$. In general, the zero locus $Z(f)|_{\{z\}\times\Sigma}$, as a divisor on $\Sigma$, converges to $Z(\gamma_d)$ at rate $1/|z|^m$. Therefore,
\[
d_k((A,\Phi^+)|_{\{z\}\times \Sigma}, (B,\sigma))
\] 
can not decay faster than $1/|z|^{m+1}$. On the other hand, the decay rate $1/|z|^m$ is also achieved:
\begin{theorem}\label{powerlaw} Suppose the polynomial map $f$ is given by
	\[
	f=\gamma_d(z^d+a_{d-1}z^{d-1}+\cdots+a_{d-m+1}z^{d-m+1})+\gamma_{d-m}z^{d-m}+\cdots
	\]
 and $\gamma_{d-m}$ is not proportional to $\gamma_d$. For the monopole $(A,\Phi^+)$ that corresponds to $(B_0,f)$ and any $k\geq 2$, there exists $C_k>0$ such that for any $z\in \C$,
	\[
d_k((A,\Phi^+)|_{\{z\}\times \Sigma}, (B,\sigma))\leq \frac{C_k}{|z|^m}.
	\]
\end{theorem}

In the generic case, $\gamma_{d-1}$ is not proportional to $\gamma_d$, so $m=1$. Theorem \ref{powerlaw} states that generically we will only have $1/|z|$ decay. The only chance to obtain exponential decay is to let $m=d$. In this case, $Z(f)$ does not change among different fibers.

\begin{theorem}\label{exponentialdecay} Suppose $f=\gamma_d\cdot f_0$ where $f_0:\C\to\C$ is a monic polynomial of degree $d$. Then for the monopole $(A,\Phi^+)$ that corresponds to $(B_0,f)$, there exists $s(k, B_0, f)$ and $C(k, B_0, f)>0$ such that
\[
d_k((A,\Phi^+)|_{\{z\}\times \Sigma}, (B,\sigma))\leq Ce^{-s|z|}.
\]
In particular, when $c_1(L^+)[\Sigma]=0$ or $1$, solutions to the Seiberg-Witten equation $(\ref{SWEQ})$ have exponential decay. 
\end{theorem}

\Remark The reason to pass to the quotient space $\SB(\Sigma, L^+)$ is to identify $(A,\Phi^+)|_{\{z\}\times \Sigma}$ with $(A,e^{i\theta}\Phi^+)|_{\{z\}\times \Sigma}$ for any $e^{i\theta}\in S^1$. In fact, if we take into account this argument, by imposing a proper gauge fixing condition for $(A,\Phi^+)$, we have for $z=|z|e^{i\theta}\in\C$,
\[
(A,\Phi^+)|_{\{z\}\times \Sigma}\sim e^{i\theta d}(B,\sigma),
\]
as $|z|\to \infty$ where $d=\deg f$. But the decay rate of their difference would depend on the gauge fixing condition. 
\medskip

In view of the previous subsection, we would expect some nice Floer theories to be developed on $Y^*$ when $c_1(L^+)=0,1$. The first reason is that in these cases, bubbles have exponential decay at infinity, as asserted in Theorem \ref{exponentialdecay}.

 The second reason is that we have a natural compactification for these bubbles. Since a degree $d$ polynomial on $\C$ is determined by its zero locus, we only need a compactification for $\Sym^d \C$ modulo translations. 

However, the situation is different when $\dim H^0(\Sigma, L^+,\bpartial_B)\geq 2$. Let us take $\gamma_1,\gamma_2\in H^0$ such that they are linearly independent. Let $t\in\C$ be a complex number and consider the family of sections 
\[
f_t(z)=\gamma_1 z+t\gamma_2. 
\]
The sequence of monopoles that correspond to $(A_0, f_t)$ does not have a good limit in any naive sense. This sequence is constructed by rescaling the $z$-coordinate on $\C$, yet the rescaling process does not preserve the metric.

Therefore, there are several natural questions to be answered based on our work:
\begin{itemize}
\item What is the compactification of the moduli space of finite energy monopoles on $X$ in general, based on Theorem \ref{main}?
\item How to compactify the moduli space of finite energy monopoles on $\R\times Y^*$? When $Y=[0,1]\times \Sigma$ and $Y^*=\R\times \Sigma$, this question is reduced to the previous one. 
\end{itemize}

\subsection{Connection to Previous Work}

This is a good time to recall classification theorems in dimension 2 and draw a comparison. The classical vortex equation on $\C$ was designed for a mathematical model of superconductors, also called the first order Ginzburg-Landau equation. Let $L\to \C$ be the trivial complex line bundle over $\C$. A configuration $(B,\gamma)\in \SC(\C,L)$ consists of a smooth unitary connection $B$ and a smooth section $\gamma$ of $L$. We set $K\equiv -1$ in (\ref{vortex}) and this term is no longer interpreted as the Gaussian curvature:
\begin{equation}\label{vortex0}
\left\{\begin{array}{r}
*iF_B+\half(|\gamma|^2-1)=0,\\
\bpartial_B\gamma=0.
\end{array}
\right.
\end{equation}

The analogous correspondence, established by Taubes \cite{Taubes} for $Y=\C$, states that 
 \begin{theorem}[\cite{Taubes, JT}]\label{1.6}There is an identification between sets
	\[
	\{\text{degree $d$ polynomials on }\C\}-\{0\}/\C^* \leftrightarrow \{\text{Vortices on $\C$ of energy $\pi d$}\}/\SG(\C).
	\]
\end{theorem}

When $Y=\Sigma$ is a compact Riemann surface, let $L$ be a complex line bundle over $\Sigma$ of degree $d:=c_1(L)$, with a Hermitian metric. In this case, the equation (\ref{vortex0}) is subject to a solvability constraint and we have a similar correspondence established by Bradlow \cite{Bradlow}:
 \begin{theorem}[{\cite[Theorem 4.3]{Bradlow}}]\label{1.8}
 When the solvability constraint $0\leq d<Vol(\Sigma)/4\pi$ is satisfied, there is a bijection between sets
 \[
 \{(\bpartial_B,f): f\neq 0\in H^0(Y,L, \bpartial_B)\}/\SG_\C(\Sigma)\leftrightarrow \{\text{Vortices on $\Sigma$ of energy $\pi d$}\}/\SG(\Sigma).
 \]
 \end{theorem}

In both cases, sets on the left are identified with the space of effective divisors of degree $d$ and isomorphic to $\Sym^d\Sigma$ and $\Sym^d\C$ respectively.

In fact, Bradlow \cite{Bradlow} defined the generalized vortex equation for any closed K\"{a}hler manifold $M$ and any Hermitian vector bundle $E^n\to M$: 
\begin{equation*}
\left\{\begin{array}{l}
i\Lambda F_B+\half \gamma\otimes \gamma^*=\half I_E\in \End(E,E)\\
\bpartial_B^2=0,\ \ \bpartial_B\gamma=0.
\end{array}
\right.
\end{equation*}
where $B\in \A(E)$ is a unitary connection and $\gamma\in \Gamma(M,E)$. The second and the third equations state that $\bpartial_B$ is integrable and $\gamma$ is holomorphic with respect to $B$. In light of Theorem \ref{coro}, Theorem \ref{main} also gives a classification for vortices on $X=\C\times \Sigma$ when $E$ is a line bundle. For details, see Section \ref{2.2}.
\smallskip

For both Theorem \ref{1.6} and \ref{1.8}, backward maps are easier to define, while constructing vortices out of holomorphic sections is hard. In \cite{Bradlow}, Bradlow proved the general existence of solutions using Kazdan-Warner's theorem \cite{KW} for any closed K\"{a}hler manifold when $\dim E=1$. However, we do not have a direct generalization of this theorem for our non-compact 4-manifold $X$. Several different proofs \cite{Bradlow2, Garcia2} to Theorem \ref{1.8} were found later. Finally, a direct gauge-theoretic proof was discovered by Garcia-Prada \cite{Garcia} where variational principle was applied to the Yang-Mills-Higgs functional. Appendix \ref{B} contains a brief review of his approach. Many of his insights have their roots in symplectic geometry, but we will not emphasis this perspective. Using his method, we will recover Taubes' theorem in Appendix \ref{C}. In \cite{Taubes}, Taubes established his theorem using variational principle on the Sobolev space $L^2_1(\C)$. Since we will work with $L^2_2(\C)$, our proof will become simpler. 

In fact, when $\Sigma$ has constant Gaussian curvature $K\equiv -1$ and $(L^+,\bpartial)$ is the trivial holomorphic line bundle over $\Sigma$, $H^0(\Sigma,L^+,\bpartial)=\C$ and we will recover Taubes' theorem from Theorem \ref{main}. 

Finally, for vortices on $\C$, there is an exponential decay result established by Taubes \cite{Taubes-thesis}:
\begin{theorem}[{\cite[P.59, Theorem 1.4]{JT}}]\label{JT}
	Let $(B,\gamma)$ be a smooth finite energy solution of the vortex equation $(\ref{vortex0})$. Given any $\epsilon>0$, there exists $M=M(\epsilon,(B,\gamma))<\infty$ such that
	\[
	0\leq * i F_B=\half(1-|\gamma|^2)<Me^{-(1-\epsilon)|z|}.
	\]
\end{theorem} 

Our proof of Theorem \ref{powerlaw} and Theorem \ref{exponentialdecay} will rely on this result. Our results, however, provide another perspective for Theorem \ref{JT}: we have exponential decay for vortices on $\C$ because for nonzero constant functions on $\Sigma$, their zero loci do not change among different fibers (since they are  empty). 
\medskip

The same classification problem is also asked for the anti-self-dual connections on the trivial $SU(2)$-bundle over $X$. Wehrheim proved an energy identity in \cite{Wehrheim}. She showed that the energy of an anti-self-dual connection, if finite, must be an integer after suitable normalization. A general classification result is still unknown. In \cite{Biquard-Jardim}, Biquard and Jardim studied finite energy instantons on $\T^2\times \C$ assuming quadratic curvature decay and proved the correspondence with algebraic objects. 

\subsection{Strategy of Proof} This paper contains several independent proofs and they could be read separately:

In Section \ref{2}, we will cover some preliminaries and prove the positivity of the analytic energy $\E_{an}$. This is not so obvious at the first glance because the Gaussian curvature shows up in (\ref{an}) and it is negative in general. As an application, we will prove Theorem \ref{coro}. In Section \ref{2.2}, we will summarize some useful facts about the vortex equation on $X$, which will be the foundation of subsequent sections.

In Section \ref{3}, we will establish the first part of Theorem \ref{main}: ``Vortices$\Rightarrow$ Polynomials". By the compactness Lemma \ref{boundedness} in Section \ref{2}, when a solution $(A,\Phi)$ is restricted to fibers $\{z\}\times\Sigma$, a subsequence will converge to a vortex on $\Sigma$. The main obstacle is to show that this limit is independent of subsequences and hence $(A,\Phi)|_{\{z\}\times\Sigma}\to (B_0,\gamma)$ as $z\to\infty$ for a fixed solution $(B,\gamma)$ to (\ref{vortex}). For this part, we will borrow ideas from \cite{Wehrheim}. 

Section \ref{4.0} is devoted to the case when $c_1(S^+)=0$. It is a simple application of maximum principle.

In section \ref{4}, we prove the second half of Theorem \ref{main}, ``Polynomials$\Rightarrow$ Vortices", by following Garcia-Prada's approach \cite{Garcia}. Our existence proof of monopoles on $X$, to a large extent, is an enhanced version of Appendix \ref{C}. To find the correct conformal factor $\alpha$, we start with an initial guess $\alpha_0$ so that
\[
(A_1,\Phi_1)\colonequals e^{\alpha_0}\cdot (\nabla_{B_0}+\frac{\partial}{\partial u}+\frac{\partial}{\partial v}, f)
\]
has finite analytic energy. A second conformal factor $\alpha_1$ is applied to minimize $\E(\alpha_1)\colonequals\E_{an}(e^{\alpha_1}\cdot(A_1,\Phi_1))$. The most technical part of the proof is an a priori estimate which allows us to control $L^2_2$ norm of $\alpha_1$ in terms of  $\E(\alpha_1)$. Thus, there is a weakly convergent subsequence in $\{\alpha_n\}$ if $\lim\E(\alpha_n)=\inf \E(\alpha)$.  We will also establish the smoothness and uniqueness of the solution.

Section \ref{sec6} is devoted to the proof of the technical estimate. It is accomplished in two steps: using the energy equation to control $\|\Delta_X\alpha\|_2$ and estimating $\|\alpha\|_2$ by decomposing $\alpha$ into high frequency and low frequency modes. Since the dimension is higher, each step here is more technical than the proof in Appendix \ref{C}. 

In Section \ref{7}, we establish the power law and exponential decay of finite energy monopoles in different cases. We prove Theorem \ref{powerlaw} and Theorem \ref{exponentialdecay}. The idea is to construct an approximating solution $\alpha_0$ using Theorem \ref{JT} and to show the correction term $\alpha_1=\alpha-\alpha_0$ has desired decay. Some elementary PDE lemmas will be used here. In the simplest form, these lemmas state that for a suitable function $u\in\SC^\infty(\R^2,\R)$ satisfying
\[
(\Delta_{\R^2}+1)u(z)=k(z)+(h\circ u)(z)
\]
where $h:\R\to\R$ is a function such that $|h(x)|<C|x|^q$ for some $q>1$ and $C>0$, the decay (power law or exponential) of the function $k: \R^2\to \R$ will produce roughly the same decay for $u$. For details, see Lemma \ref{formallemma} and Lemma \ref{formallemma2}. 

Some analytic results are collected in Appendix \ref{A}. We state a weak version of Trudinger's inequality that is used elsewhere in this paper.

\textbf{Acknowledgments.} The author is extremely grateful to his advisor, Tom Mrowka, for suggesting this project and for his invaluable support. The author would like to thank Ao Sun and Jianfeng Lin for valuable discussions and comments. This material is based upon work supported by the National Science Foundation under Grant No. 1808794. 

\section{Monopoles $\Leftrightarrow$ Vortices}\label{2}

In this section, we prove some basic properties of finite energy monopoles. In Section \ref{2.1},  we describe \spinc structures on $X=\C\times\Sigma$ and establish the positivity of the analytic energy $\E_{an}$. In Theorem \ref{coro}, it is shown that monopoles with $\E_{an}<\infty$ on $X$ are degenerate, in the sense that either $\Phi_+\equiv 0$ or $\Phi_-\equiv 0$. Once this reduction to vortices is made, we will not work with \spinc structures any longer. In Section \ref{2.2}, we list several useful facts about the vortex equation on $X$  which form the foundation of later sections. 

The same reduction was also proved for the monopole equations on a Seifert-fibered 3-manifold $Y$ and $\R\times Y$ in \cite{Mrowka-Ozsvath-Yu}. Our proof is parallel to the vanishing spinor argument in \cite[Theorem 4 \& Proposition 6.5]{Mrowka-Ozsvath-Yu}.

\subsection{Preliminaries}\label{2.1}

Since $X=\C\times \Sigma$ is a complex manifold, it is endowed with the complex orientation. A \spinc structure of $X$ can be described concretely. The decomposition $S^+=L^+\oplus L^-$ is parallel, so any \spinc connection $A$ must split as 
\[
\nabla_A=\begin{pmatrix}
\nabla_{A_+} & 0 \\
0 & \nabla_{A_-}\\
\end{pmatrix},
\]
where $A_\pm$'s are unitary connections of $L_\pm$. Let $z=u+iv$ be the coordinate function on $\C$. The Clifford multiplication $\rho=\rho_4: T^*X\to \Hom(S,S)$ can be constructed by setting: 
\begin{eqnarray*}
\rho_4(du)=\begin{pmatrix}
0& -id\\
id & 0\\
\end{pmatrix},\ 
\rho_4(dv)=\begin{pmatrix}
0& \sigma_1\\
\sigma_1 & 0\\
\end{pmatrix} :\ S^+\oplus S^-\to S^+\oplus S^-,
\end{eqnarray*}
where $\sigma_1=\begin{pmatrix}
i & 0\\
0 & -i\\
\end{pmatrix}: S^+\to  S^+$ is the first Pauli matrix.  The bundle $L^-$ is isomorphic to $L^+\otimes \bigwedge^{0,1}\Sigma$ and under this identification,  
\[
\rho_3(w)\colonequals\rho_4(du)^{-1}\cdot\rho_4(w)= \begin{pmatrix}
0 & -\iota(\sqrt{2}w^{0,1})\ \cdot \\
\sqrt{2}w^{0,1}\otimes \cdot  & 0\\
\end{pmatrix}: S^+\to S^+,
\]
for any $x\in \Sigma$ and $w\in T_x\Sigma$. This relation defines $\rho_4(w)$. $\rho_3$ is used to denote the 3-dimensional Clifford multiplication.
\medskip

We regard $L^+$ and $L^-$ as bundles on $\Sigma$, and they pull back to line bundles over $X$ via the projection map $X\to \Sigma$. Choose a unitary connection $B_+$ on $L^+\to \Sigma$. $B_+$  and the Levi-Civita connection on $\bigwedge^{0,1}\Sigma$ induce a unitary connection $B_-$ on the line bundle $L^-=L^+\otimes \bigwedge^{0,1}\Sigma$. We obtain a background connection $A_0$ on $S^+$ by the setting
\[
\nabla_{A_0}=\nabla_{B_0}+\frac{d}{du}+\frac{d}{dv},
\]
where $B_0=(B_+,B_-)$ is the unitary connection on $S^+\to \Sigma$. One can easily check that $A_0$ is a \spinc connection. Any other \spinc connection $A$ differs from $A_0$ by an imaginary 1-form $a\in \Gamma(X,iT^*X)$. Their curvature tensors are related by
\[
F_A=F_{A_0}+da\otimes \Id. 
\]

Using the product structure on $X$,  the covariant derivative $\nabla_A=(\nabla^\C_A,\nabla_A^\Sigma)$ is decomposed into $\C$-direction part and $\Sigma$-direction part. The curvature tensor $F_A$ is decomposed accordingly as:
\[
F_A=F_A^\Sigma+F_A^\C+F^m_A,
\]
where $F_A^m$ is the mixed term. Let $A^t$ be the induced connection on $\bigwedge^2 S^+=L^+\otimes L^-$. A Similar decomposition applies to the curvature form $F_{A^t}$ of $A^t$:  
 \[
F_{A^t}=F_{A^t}^\Sigma dvol_\Sigma+ F_{A^t}^\C dvol_\C+F_{A^t}^m,
\]
where $F_{A^t}^m\in \Gamma(X,i\Omega^1(\C)\wedge\Omega^1(\Sigma))$. Our description of $F_A$ then shows
\begin{equation}\label{22}
F^m_A=\half F_{A^t}^m\otimes \Id,
\end{equation}
and 
\begin{equation}\label{23}
F^\Sigma_A=\begin{pmatrix}
F_{A_+}^\Sigma & 0 \\
0 & F_{A_-}^\Sigma
\end{pmatrix}dvol_{\Sigma}=\begin{pmatrix}
\half F_{A^t}^\Sigma+\frac{i}{2} K& 0 \\
0 & \half F_{A^t}^\Sigma-\frac{i}{2} K
\end{pmatrix}dvol_{\Sigma}.
\end{equation}

In particular, we obtain that
\begin{equation}\label{chernclass}
c_1(S^+)=c_1(\bigwedge\nolimits^{\!2}  S^+)[\Sigma]=2(c_1(L^+)+1-g)=2(c_1(L^-)-1+g).
\end{equation}

To prove Theorem \ref{coro}, we need a more useful expression of $\E_{an}$. The following lemma establishes the positivity of the analytic energy:
\begin{lemma}\label{positivity}
Over each fiber $\{z\}\times \Sigma$, we have energy identity:
\begin{align*}
\int_\Sigma \frac{1}{4}|F_{A^t}|^2+|\nabla_A\Phi|^2+\frac{1}{4}|\Phi|^4+ \frac{K}{2}|\Phi|^2&=\int_\Sigma \frac{1}{4}|F_{A^t}^\C|^2+\frac{1}{4} |F_{A^t}^m|^2+|\nabla_A^\C\Phi|^2\\
&+\int_\Sigma |\Phi_+|^2|\Phi_-|^2+|D_A^\Sigma\Phi|^2\\
&+\int_\Sigma \frac{1}{4}|iF_{A^t}^\Sigma+|\Phi_+|^2-|\Phi_-|^2|^2,
\end{align*}
where $D_A^\Sigma=\sum_{i=1,2}\rho_3(e_i)\nabla_{A,e_i}$ is the Dirac operator on $\Sigma$. Here, $\{e_1,e_2\}$ is any orthonormal frame at some point $p\in \Sigma$. 
\end{lemma}
In particular, Lemma \ref{positivity} implies that the analytic energy $(\ref{an})$ is always a non-negative number. 
\begin{proof} The Dirac operator $D_A^\Sigma$ interchanges bundles $L^+$ and $L^-$. In other words, we have
\[
D_A^\Sigma=
\begin{pmatrix}
0 & D^-\\
D^+ & 0\\
\end{pmatrix}.
\] 
Under the isomorphism $L^-\cong L^+\otimes \bigwedge^{0,1}\Sigma$ and $L^+\cong L^-\otimes \bigwedge ^{1,0}\Sigma$, $D^+$ and $D^-$ are written as
\[
D^+=\sqrt{2}\bpartial_{A_+}\colonequals \sqrt{2}(\nabla_{A_+})^{0,1},\  D^-=\sqrt{2}\partial_{A_-}\colonequals \sqrt{2}(\nabla_{A_-})^{1,0}. 
\]

Therefore, it is sufficient to prove
\begin{align*}
\int_\Sigma |D^+\Phi_+|^2&=\int_\Sigma |\nabla_{A_+}^\Sigma\Phi_+|^2+\langle \Phi, \half(K-iF^\Sigma_{A^t})\Phi\rangle,\\
\int_\Sigma |D^-\Phi_-|^2&=\int_\Sigma |\nabla_{A_-}^\Sigma\Phi_-|^2+\langle \Phi, \half(K+iF^\Sigma_{A^t})\Phi\rangle.
\end{align*}

By (\ref{23}),  $F_{A_+}^\Sigma=\half (F^\Sigma_{A^t}+iK)$ and $F_{A_-}^\Sigma=\half (F^\Sigma_{A^t}-iK)$. At this stage, we apply Weitzenb\"{o}ck formulas:  for any line bundle $L\to \Sigma$, a unitary connection $B$ and a section $\sigma\in \SC^\infty(\Sigma,L)$, we mush have
\begin{align*}
2\int_\Sigma |\bpartial_B\sigma|^2&=\int_\Sigma |\nabla_B\sigma|^2- \langle \sigma, iF_B\sigma\rangle, \\
2\int_\Sigma |\partial_B\sigma|^2&=\int_\Sigma |\nabla_B\sigma|^2+\langle \sigma, iF_B\sigma\rangle. 
\end{align*}
\end{proof}

\begin{lemma}\label{boundedness}
If $(A,\Phi)$ is any smooth solution to (\ref{SWEQ}) with $\E_{an}(A,\Phi)<\infty$, there is a constant $C=C(\E_{an}(A,\Phi))>0$ such that for any $z\in \C$,
\[
\int_{\{z\}\times \Sigma}|\Phi|^2<C. 
\]
\end{lemma}

\begin{proof}
This is a consequence of Lemma \ref{positivity} and the classical compactness theorem. Let $n=(n_1,n_2)\in \Z\times\Z\subset \C$. Then for $(z,x)\in X'\colonequals\overline{B(0, 10)}\times \Sigma\subset X$, set 
 \[(A_n,\Phi_n)(z,x)\colonequals(A,\Phi)(z-n,x).
 \] 
 
Then $(A_n,\Phi_n)$ solves $(\ref{SWEQ})$ on $X'$. In light of Lemma \ref{positivity}, 
 \begin{align*}
\E_{an}(A_n,\Phi_n)&\colonequals\int_{X'} \frac{1}{4}|F_{A_n}|^2+|\nabla_{A_n}\Phi|^2+\frac{1}{4}|\Phi_n|^4+ \frac{K}{2}|\Phi_n|^2\leq \E_{an}(A,\Phi).
 \end{align*}
 
By \cite[Theorem 5.1.1]{Bible}, after proper gauge transformations, a subsequence of $(A_n,\Phi_n)$ will converge in $\SC^\infty$-topology in the interior. This shows 
\[
\|\Phi_n\|_{L^\infty(B(0,5)\times\Sigma)}
\]
is uniformly bounded by some constant $C>0$. It is clear that $C$ can be made to be independent of $(A,\Phi)$ and to depend only on $\E$. 
\smallskip

Finally, let $(A_\infty,\Phi_\infty)$ be the limit of this subsequence on $B(0,5)\times \Sigma$. Then $(A,\Phi)$ solves equation (\ref{SWEQ}) and its analytic energy is zero. By the computation in Section \ref{3}, $(A,\Phi)$ is a constant family of vortices on $B(0,5)$. In other words, up to a gauge transformation, 
\[
(A_\infty,\Phi_\infty)=(\nabla_A^\Sigma+\frac{d}{du}+\frac{d}{dv},\sigma),
\]
where the pair $(\nabla_A^\Sigma,\sigma)$ is independent of $z\in\C$. This observation will be useful later. In fact, it is the fundamental issue to be resolved in Section \ref{3} that this limit is independent of the subsequence we choose.
\end{proof}

Now we have the ammunition to attack Theorem \ref{coro}. We start with a weak statement. 

\begin{proposition}
	If $(A,\Phi)$ is any smooth solution to the monopole equation $(\ref{SWEQ})$ on $X$ with $\E_{an}(A,\Phi)<\infty$, then either $\Phi_+\equiv 0$ or $\Phi_-\equiv 0$. 
\end{proposition}

\begin{proof} Define a complex-valued function $G$ on $\C$ by the formula
\begin{align*}
G(z)=\int_{\{z\}\times \Sigma} \langle D^+ \Phi_+,\Phi_- \rangle,
\end{align*}
for any $z\in \C$. We compute $\bpartial G$:
\begin{lemma}\label{partialG} There is an identity:
\[
\bpartial G(z)=-\int_{\{z\}\times \Sigma} (|D^+\Phi_+|^2+|D^-\Phi_-|^2+|\Phi_+|^2|\Phi_-|^2).
\]
In particular, $\bpartial G$ is real and non-positive.
\end{lemma}
\begin{proof} Let $\frac{D}{du}$ and $\frac{D}{dv}$ denote the covariant derivative $\nabla_{A,\frac{\partial}{\partial u}}$ and $\nabla_{A,\frac{\partial}{\partial v}}$ respectively. The second equation of $(\ref{SWEQ})$ implies
\[
\frac{D}{du}\Phi+\sigma_1\cdot \frac{D}{dv}\Phi+ D_A^\Sigma \Phi=0.
\]
In particular, this shows 
\begin{align}\label{10}
(\Du+i\Dv)\Phi_+=-D^-\Phi_-,\ (\Du-i\Dv)\Phi_-=-D^+\Phi_+.
\end{align}

Since $D^+$ is the adjoint of $D^-$ as operators on $L^2(\Sigma)$, we have 
\begin{align*}
\bpartial G(z)&=(\frac{\partial}{\partial u}+i\frac{\partial}{\partial v})\int_{\{z\}\times \Sigma} \langle D^+ \Phi_+,\Phi_- \rangle\\
&= \int_{\{z\}\times \Sigma} \langle (\Du+i\Dv)D^+ \Phi_+,\Phi_- \rangle+\int_{\{z\}\times \Sigma} \langle D^+ \Phi_+,(\Du-i\Dv)\Phi_- \rangle\\
&=-\int_{\{z\}\times \Sigma} (|D^+\Phi_+|^2+|D^-\Phi_-^2|^2)+\int_{\{z\}\times\Sigma}\langle [\Du+i\Dv, D^+]\Phi_+,\Phi_-\rangle. 
\end{align*}

By the formula $D_A^\Sigma=\rho_3(e_i)\nabla_{A,e_i}$ and (\ref{22}), the commutator can be computed as 
\begin{align*}
\qquad &(\Du-\sigma_1 \Dv) D_A^\Sigma-D_A^\Sigma (\Du+\sigma_1 \Dv)\\
&=\sum_{i=1,2}\rho_3(e_i)[\Du, \nabla_{A,e_i}]-\sigma_1\cdot \rho_3(e_i)[\Dv, \nabla_{A,e_i}]\\
&= \sum_{i=1,2}\rho_3(e_i)F^m_A(\frac{\partial}{\partial u}, e_i)-\sigma_1\cdot \rho_3(e_i)F^m_A(\frac{\partial}{\partial v}, e_i)\\
&=\half\sum_{i=1,2} \rho_3(e_i)F_{A^t}(\frac{\partial}{\partial u}, e_i)-\sigma_1\cdot \rho_3(e_i)F_{A^t}(\frac{\partial}{\partial v}, e_i)\\
&=-\half \rho_4(F_{A^t}^m)|_{S^+}=-\half \rho_4((F_{A^t}^m)^+)=-(\Phi\Phi^*)_\Pi. 
\end{align*}
At the last step, we used the first equation of (\ref{SWEQ}). Here, $\Pi$ denotes the projection map from a 2 by 2 matrix to its off-diagonal part. Therefore, 
\[
\int_\Sigma\langle [\Du+i\Dv, D^+]\Phi_+,\Phi_-\rangle=-\int_\Sigma\langle (\Phi_-\Phi_+^*)\Phi_+,\Phi_-\rangle=-\int_\Sigma |\Phi_+|^2|\Phi_-|^2. \qedhere
\]
\end{proof}

Write $G=X+iY$ with $X,Y$ real. Then Lemma \ref{partialG} implies
\[
\partial_u X-\partial_vY\leq 0, \partial_u Y+\partial_v X=0. 
\]

Set $K(z)=\int_0^z Xdu-Ydv$. By the second equation, this integral is independent of the path we choose. Therefore,
\[
X=\partial_uK, Y=-\partial_v K,
\]
and 
\[
\Delta_\C K=(-\partial_u^2-\partial_v^2)K=-\bpartial G\geq 0. 
\]
By Lemma \ref{boundedness} and the Cauchy-Schwartz inequality, we have
\begin{align*}
|\nabla K|^2=|G|^2\leq \|D^+\Phi_+\|_{L^2(\Sigma)}^2\|\Phi_-\|_{L^2(\Sigma)}^2 \leq C \Delta K.
\end{align*}

Our goal is to show $K\equiv 0$. Let $Z(r)\colonequals \int_{\partial B(0,r)} \Delta K\geq 0$. Then integration by parts shows
\begin{multline*}
0\leq W(R)\colonequals\int_0^R Z(r)dr=\int_{B(0,R)} \Delta K= \Bigl|\int_{\partial B(0,R)} \vec{n}\cdot \nabla K\Bigr|\\
\leq (2\pi R)^\half (\int_{\partial B(0,R)} |\nabla K|^2)^\half \leq  (2\pi CR\int_{\partial B(0,R)} \Delta K)^\half\leq  C_2R^\half Z(R)^\half. 
\end{multline*}

Suppose $W(r_0)>0$ for some $r_0$. Then for $r>r_0$, 
\[
\ln' (r)\leq C_3\Bigl(-\frac{1}{W}\Bigr)',
\]
and hence for any $r_1>r_2>r$, 
\[
\ln(r_1)-\ln(r_2)\leq C_3\Big(\frac{1}{W(r_2)}-\frac{1}{W(r_1)}\Big). 
\]

Therefore, $W(r)$ must blow up in finite time if $W(r)\not\equiv 0$. Hence, $\Delta K\equiv 0$ and
\[
D^+\Phi_+\equiv 0,\ D^-\Phi_-\equiv 0,\ |\Phi_+||\Phi_-|\equiv 0. 
\]

This shows over each fiber, $\Phi_+$ and $\Phi_-$ are either holomorphic or anti-holomorphic with respect to some connections. They have discrete zero locus unless the whole section is zero. Therefore, either $\Phi_+$ or $\Phi_-$ is zero over that fiber. By (\ref{10}), they are also holomorphic or anti-holomorphic on $\C\times\{x\}$ for any $x\in\Sigma$, so one of them is identically zero on $X$.  
\end{proof}
To prove Theorem \ref{coro}, it remains to verify that if any finite energy monopole exists, then we have the constraint $0<|c_1(S^+)|<2g-2$ and the sign of $c_1(S^+)$ will determine which of $\Phi_+$ and $\Phi_-$ vanishes.  
\begin{proof}[Proof of Theorem \ref{coro}]
In light of Lemma \ref{boundedness}, a subsequence of $(A_n,\Phi_n)$ will converge to $(A_\infty,\Phi_\infty)$ on $B(0,5)\times \Sigma$ and the energy of this limit vanishes: $\E_{an}(A_\infty,\Phi_\infty)=0$. By Lemma \ref{positivity}, for this limit, we must have
\[
iF_{A^t_\infty}^\Sigma+|\Phi_{\infty, +}|^2-|\Phi_{\infty, -}|^2\equiv 0.
\]
If $c_1(S^+)<0$ and $\Phi_+\equiv 0$, then integrating over $\Sigma$ yields a contradiction:
\[
0\geq -\int_{\Sigma}|\Phi_{\infty,-}|^2= -\int_{\Sigma} iF_{A^t_\infty}^\Sigma=-2\pi c_1(S^+)>0. 
\]
 
Therefore, $c_1(S^+)<0$ implies $\Phi_-\equiv 0$. Since $D^+\Phi_{+,\infty}\equiv 0$ and this section is nonzero, we must have $c_1(L^+)\geq 0$. By (\ref{chernclass}), this forces $c_1(S^+)\geq 2-2g$. The case when $c_1(S^+)>0$ is dealt with similarly.  
\end{proof}

\subsection{Vortices and the energy equation}\label{2.2}
From now on, we will assume $2-2g\leq c_1(S^+)<0$ and $\Phi_-\equiv 0$. For simplicity, we will change our notation. Let $L=L^+$ and $\sigma=\Phi_+$. We will use $A$ to denote a unitary connection on $L$ and use $\hat{A}$ for the induced \spinc connection on $S^+\cong L\oplus (L\otimes \bigwedge^{0,1}\Sigma)$. Recall that the curvature form $F_A$ is divided into three parts:
\[
F_A=F_A^\Sigma dvol_\Sigma+F_A^\C dvol_\C+F^m_A,
\]
where $F^m_A\in \Gamma(X,i\Omega^1(\C)\wedge\Omega^1(\Sigma))$ is the mixed term. Then the monopole equation (\ref{SWEQ}) is simplified as 
\begin{subequations}\label{truncatedmonopole}
\begin{align}
i(F_A^\Sigma+F_A^\C)+\half K+\half |\sigma|^2&=0,\label{firsteq}\\
\bpartial_A\sigma&=0, \label{secondeq}\\
F^m_A&\in \Lambda^-(X).\label{thirdeq}
\end{align}
\end{subequations}

This is precisely the vortex equation on $X$ (compare \cite{Bradlow}). The last equation (\ref{thirdeq}) is equivalent to $\bpartial_A^2=0$, or $F^{0,2}_A=0$, i.e., $\bpartial_A$ is integrable. Since $\nabla_{A}=(\nabla_A^\C,\nabla_{A}^\Sigma)$, the second equation (\ref{secondeq}) is equivalent to two equations: 
\[
\bpartial_A^\C\sigma=0=\bpartial_A^\Sigma \sigma.
\]

By Lemma \ref{positivity}, for a smooth configuration $(A,\sigma)$, its analytic energy is given by the formula,
\begin{align}\label{analyticenergy}
\E_{an}(A,\sigma)&=\int_X  |iF_A+\frac{1}{2}Kdvol_\Sigma|^2+|\nabla_A\sigma|^2+\frac{1}{4}|\sigma|^4+ \frac{K}{2}|\sigma|^2,\\
&=\int_X  |F_A^\C|^2+|F^m_A|^2+|\nabla_A^\C\sigma|^2+2|\bpartial_A^\Sigma\sigma|^2\nonumber\\
&\qquad+|iF_A^\Sigma+\half K+\half|\sigma|^2|^2.\nonumber
\end{align}

The energy formula in below concerns the analytic energy on a compact region $X_r\colonequals B(0,r)\times\Sigma\subset X$. This is just the energy equation for the Seiberg-Witten map (see \cite[Proposition 4.5.2]{Bible}), but this particular expression will be convenient to use: 
\begin{lemma}\label{energyball} Let $F_A'= F_A-i\frac{K}{2}dvol_\Sigma$. Define 
	\[
	\E(r)= \int_{X_r} |F_A'|^2+|\nabla_A\sigma|^2+\frac{1}{4}|\sigma|^4+ \frac{K}{2}|\sigma|^2.
	\]
	Suppose the configuration $(A,\sigma)$ satisfies $(\ref{secondeq})$ and $ (\ref{thirdeq})$, then
	\begin{align*}
	\E(r)&=\int_{X_r}|i(F_A^\Sigma+F_A^\C)+\half K+\half |\sigma|^2|^2+\int_{\partial X_r} \langle \sigma, \nabla_{A,\vec{n}}\sigma\rangle+ \int_{X_r} F_A'\wedge F_A'. 
	\end{align*}
\end{lemma}
\begin{proof} We expand the bracket:
\begin{align*}
|i(F_A^\Sigma+F_A^\C)+\half K+\half |\sigma|^2|^2&=  |iF_A^\Sigma+\half K+\half |\sigma|^2|^2+2\langle iF_A^\C,\half|\sigma|^2\rangle\\
&\qquad+|F_A^\C|^2+2\langle iF_A^\C, \half K+iF_{A}^\Sigma\rangle.  
\end{align*}

\textit{Step 1.} The Weitzenb\"{o}ck formula shows that over $\C$:
\begin{equation}\label{overC}
0=2(\bpartial_A^\C)^*\bpartial_A^\C\sigma=(\nabla_A^\C)^*\nabla_A^\C\sigma-iF_A^\C \sigma.
\end{equation}

Take inner product with $\sigma$ and do integration by parts:
\begin{equation*}
\int_{X_r} iF_A^\C|\sigma|^2=\int_{X_r}|\nabla_A^\C\sigma|^2-\int_{\partial X_r}\langle \sigma, \nabla_{A,\vec{n}}\sigma\rangle.
\end{equation*}

\textit{Step 2.} Since $F^m_A$ is imaginary and anti-self-dual, we have
\begin{align*}
\int_{X_r} F_A'\wedge F_A'&=2\int_{X_r} F_A^\C(F_A^\Sigma-\frac{i}{2}K)dvol_X-\int_{X_r}F^m_A\wedge *F^m_A\\
&=-2\int_{X_r} \langle iF_A^\C, iF_A^\Sigma+\frac{1}{2}K\rangle dvol_X+\int_{X_r}|F^m_A|^2.
\end{align*}

Now we use Lemma \ref{positivity} or  formula (\ref{analyticenergy}) to conclude.
\end{proof}

For the rest of the paper, we will not work with \spinc structures and the Seiberg-Witten equation (\ref{SWEQ}); at least, not in a direct way. Instead, the equation (\ref{truncatedmonopole}) will become the main object of study.

\section{Vortices $\Rightarrow$ Polynomials}\label{3}

It is the object of this section to show any smooth solution $(A,\sigma)$ to equation (\ref{truncatedmonopole}) with finite energy comes from a holomorphic line bundle $\SL$ and a polynomial map $f$ provided that $0\leq c_1(L)<g-1$. This proves one direction of Theorem \ref{main}. First, we recall some definitions.

Let $\SL=(L,\bpartial_\SL)$ be a holomorphic structure  on $L\to\Sigma$. Then $H^0(\Sigma, \SL)$ is a complex vector space of finite dimension. The Chern connection on $\SL\to\Sigma$ is the unique unitary connection $B_0=\nabla_0$ such that
\[
\nabla^{0,1}_0=\bpartial_{\SL}. 
\]

Note that $\SL$ pulls back to a holomorphic line bundle on $X$. By abuse of notation, we still denote it by $\SL$. A polynomial map $f:\C\to H^0(\Sigma, \SL)$ is regarded as a holomorphic section $\sigma_0$ of $\SL\to X$ by setting $\sigma_0(z,x)=f(z)(x)$. 
The connection $B_0$ also induces a unitary connection on $X$ by the formula:
\[
\nabla_{A_0}=\nabla_0+\frac{\partial}{\partial u}+\frac{\partial}{\partial v},
\]
and $\bpartial_{A_0}=\bpartial_\SL$. 

An element in the configuration space $\SC(X, L)$ consists of a pair $(A,\sigma)$ where $A$ is a smooth unitary connection of $L$ and $\sigma\in \Gamma(X,L)$ is a smooth section. Thus,  $\SC(X, L)=\A(X,L)\times \Gamma(X,L)$. The gauge group $\SG_\C(X)=\map(X,\C^*)=\SG(X)\times \Conf(X)$ acts on $\SC(X,L)$ by the formula 
\begin{align}\label{gaugetransformation}
g=u\cdot e^{\alpha}: (A,\sigma)&\mapsto( A+i*_\C d_\C\alpha+i*_\Sigma d_\Sigma\alpha-u^{-1}du,u\cdot e^{\alpha}\sigma),
\end{align}
for any $\alpha\in \SC^\infty(X,\R)$ and $u\in \SC^\infty(X,S^1)$. Whenever a subscript $\C$, $\Sigma$ or $X$ is used, it denotes the operator on corresponding manifolds. For instance, $d_\C$ and $*_\C$ denote the exterior differential and the Hodge $*$-operator on $\C$. The same holds for $\Sigma$. 

When $u\in \SG(X)$ the gauge action on $\A(X,L)$ is defined by pulling back connections:
\[
\nabla_{u(A)}=u\nabla_{A}(u^{-1}\cdot ).
\]
But this is not the case when $e^\alpha\in \Conf(X)$. In fact, (\ref{gaugetransformation}) is designed by requiring two properties:
\begin{enumerate}
\item $u(A)$ is a unitary connection. In other words, $u(A)-A$ is an imaginary  1-form on $X$. 
\item $\bpartial_{u(A)}\sigma =u\bpartial_A (u^{-1}\sigma)$. That is to say, the $(0,1)$-part of $u(A)$ is the pull back of the $(0,1)$-part of $A$. 
\end{enumerate}

Under the action of $\Conf(X)$, the curvature form and the covariant derivative are changed by the formula:
\begin{align}\label{covariantderivative}
g=e^\alpha: F_A&\mapsto F_A-(i\Delta_\C\alpha) dvol_\C-(i\Delta_\Sigma\alpha) dvol_\Sigma+F^-(\alpha),\\
\nabla_A\sigma&\mapsto e^{\alpha}(\nabla_A\sigma+2(d\alpha)^{1,0}\otimes \sigma), \nonumber
\end{align}
where $F^-(\alpha)$ reflects the change of the mixed term (the $F^m_A$-part). It lies in $\Lambda^-(X)\subset \Lambda^2(X)$. Indeed,
\[
F^-(\alpha)=i(-*_\C+*_\Sigma)d_\C d_\Sigma\alpha,
\] 
and $*_XF^-(\alpha)=-(*_\C*_\Sigma)F^-(\alpha)=-F^-(\alpha)$. This shows that for any $\alpha\in \Gamma(X,\C)$, if a configuration $(A,\sigma)$ satisfies equations (\ref{secondeq}) and (\ref{thirdeq}), so does $e^\alpha\cdot (A,\sigma)$. 

The main result of this section is the following theorem:
\begin{theorem}\label{polynomial}
Suppose $0\leq c_1(L)<g-1$ and $(A,\sigma)$ is any smooth solution to $(\ref{truncatedmonopole})$ with finite analytic energy, then there is a complex gauge transformation $e^\alpha$ with $\alpha\in \SC^\infty(X,\C)$ such that 
\[
e^\alpha\cdot (A,\sigma)=(A_0, \sigma_0), 
\]
where $(A_0,\sigma_0)$ is the configuration induced from some pair ($\SL, f$). The pair $(\SL, f)$ is unique up to complex gauge transformation on $\Sigma$.   
\end{theorem}

The case when $c_1(L)=g-1$, i.e. $c_1(S^+)=0$ is dealt with in the next section.
\begin{proof} Write $\nabla_A$ as
\[
\nabla_A =\nabla_B^\Sigma+\frac{\partial}{\partial u}+hdu+\frac{\partial}{\partial v}+gdv,
\]
where for each $z\in \C$, $B(z)$ is a unitary connection on $L\to \Sigma$ and $g,h\in\SC^\infty(X,i\R)$ are smooth functions. Then 
\[
F_A=F_B^\Sigma dvol_\Sigma+(\frac{\partial g}{\partial u}-\frac{\partial h}{\partial v})du\wedge dv+ du\wedge (\frac{\partial B}{\partial u}-d_\Sigma h)+dv\wedge (\frac{\partial B}{\partial v}-d_\Sigma g).
\]

 We start by analyzing the equation (\ref{thirdeq}). In light of the decomposition above, this equation is equivalent to
\begin{equation}\label{hodge}
(\frac{\partial }{\partial u}+*_\Sigma\frac{\partial }{\partial v})B=d_\Sigma h+*_\Sigma d_\Sigma g.
\end{equation}

Suppose a background unitary connection $B_0$ on $L$ is chosen. Then in terms of Hodge decomposition, we have 
\begin{equation*}
B(z)-B_0=b^1(z)+b^h(z)+b^2(z),
\end{equation*}
where $b^1, b^h$ and $b^2$ are imaginary exact, harmonic and co-exact 1-forms on $\Sigma$ respectively. We impose the following gauge fixing condition:
\begin{flalign}\label{GFC}
 b^1\equiv 0. & 
\end{flalign}

This can be achieved since $b^1(z)=id_\Sigma \beta(z)$ for a unique function $\beta(z)\in (\ker\Delta_\Sigma)^\perp$ where $\Delta_\Sigma$ is the Hodge Laplacian operator on $\Sigma$ and $(\ker\Delta_\Sigma)^\perp$ denotes the $L^2$-orthogonal complement of the kernel. The function $\beta\in \SC^\infty(X, \R)$ is smooth since $b^1$ is. Then we can work instead with $e^{i\beta}\cdot (A,\sigma)$. 
 
 Suppose the gauge fixing condition (\ref{GFC}) has been imposed. Equation (\ref{hodge}) implies
\[(\frac{\partial }{\partial u}+*_\Sigma\frac{\partial }{\partial v})b^h=0.
\]
 
 On the other hand, by identity (\ref{analyticenergy}), 
 \[
 \|\nabla b^h\|_{L^2(\C)}^2\leq  \|F^m_A\|^2_{L^2(X)}\leq \E_{an}(A,\sigma)<\infty. 
 \]
 
This shows that the function $b^h:\C \to (H^1(\Sigma, i\R),*_\Sigma)$ is holomorphic and its derivative lies in $L^2(\C)$. Therefore, $b^h$ is a constant function on $\C$. By changing the background connection, we assume $b^h\equiv 0$. 
 
Now it remains to analyze $b^2$. We know $b^2(z)=-i*_\Sigma d_\Sigma \alpha(z)$ for a unique function $\alpha(z)\in (\ker\Delta_\Sigma)^\perp$. The function $\alpha\in\SC^\infty(X, \R)$ is smooth. By comparing exact and co-exact parts of the equation (\ref{hodge}), we have
 \[
 d_\Sigma(\frac{\partial }{\partial u} i\alpha+g)=0,\ d_\Sigma(\frac{\partial }{\partial v} i\alpha-h)=0,
\]
so
\begin{equation}\label{omega}
\nabla_A =\nabla_B+\frac{\partial}{\partial u}+\frac{\partial}{\partial v}-i*_\C d_\C\alpha+\omega. 
\end{equation}
for some imaginary $1$-form $\omega\in \Gamma(\C, iT^*\C)$. Let $(A',\sigma')=e^{\alpha}\cdot (A,\sigma)$. By (\ref{gaugetransformation}), we have 
\[
\nabla_{A'}=\nabla_{0}+\frac{\partial}{\partial u}+\frac{\partial}{\partial v}+\omega.
\]

Let $\nabla_\omega=\nabla_{A'}^\C$. Then the equation (\ref{secondeq}) implies $
\bpartial_{B_0}\sigma'(z)=0$
for each $z\in \C$. Thus, we obtain a map
\[
f':\C\to H^0(\Sigma,L, \bpartial_{B_0}),
\]
and $\bpartial_\omega f'\equiv 0$ by (\ref{secondeq}). 

At this moment, it suffices to show that we can eliminate $\omega$ by applying a further conformal transformation and obtain $(\nabla_\omega,f')$ from the trivial connection and a polynomial map. However, it is hard to do this directly. The main obstacle is to verify the following property:
\begin{lemma}\label{finitezeros}
There exists  a section $v\neq 0\in H^0(\Sigma,\SL,\bpartial_{B_0})$, a positive number $c>0$ and a sequence of numbers $r_{n+1}>r_n>0$ with $\lim_{n\to \infty}r_n= 0$ such that $f_1=\langle f', v\rangle$ has finitely many zeros on $\C$ and $|f_1(z)|>c$ for any $z\in \partial B(0,r_n)$. 
\end{lemma}
This lemma is hard because we need to connect the finiteness of zeros of $f_1$ with the finiteness of $\E_{an}$. It is not clear to the author whether there is a clean and straightforward solution. In fact, by (\ref{analyticenergy}), the finiteness of $\E_{an}$ implies $F_\omega, \partial_\omega f'\in L^2(\C)$. 
\begin{question}
Suppose $d+\omega$ is a unitary connection of the trivial line bundle on $\C$ and $f': \C\to\C^n$ is a holomorphic section with respect to $\omega$, i.e. $\bpartial_\omega f'=0$. If $F_\omega, \nabla_{\omega}f'\in L^2(\C)$, then there exists a real valued function $\alpha\in\Gamma(X,\R)$ such that $e^\alpha\cdot (\nabla_\omega, f')=(d,f_0)$ where $d$ is the exterior differential and $f_0$ is a polynomial map.  
\end{question}

The author does not know if this question could be answered by using complex analysis of one variable.  Lemma \ref{finitezeros} is easily proven when $\dim H^0(\Sigma, \SL, \bpartial_{B_0})=1$ since in this case the vortex moduli space is a single point and we conclude by the compactness argument used in Lemma \ref{boundedness}. 

We shall prove Theorem \ref{polynomial} assuming Lemma \ref{finitezeros}. The proof of Lemma \ref{finitezeros} is postponed to the end of section. The following lemma is a direct consequence of the proof of Lemma \ref{boundedness}:

\begin{lemma}\label{boundedness2}
The sections $\alpha, \sigma\in L^\infty_k(X)$ for any $k\geq 0$. In particular, $f'\in L^\infty(\C)$. 
\end{lemma}

Take a complex gauge transformation $u=e^{\alpha_1+i\beta_1}$ with $\alpha_1,\beta_1\in \SC^\infty(\C,\R)$. Consider $u\cdot (\omega,f_1)$. The connection form $\omega$ is changed into
\[
\omega+i(*_\C d_\C \alpha_1-d_\C\beta_1).
\]

To make it zero, we need to solve the equation
\begin{equation}\label{holo}
\bpartial (\alpha_1+i\beta_1)=\omega^{0,1}. 
\end{equation}

In general, this equation can not be solved on $\C$. But the $\bpartial$-Poincar\'{e} lemma says that we can alway solve it on $B(0,2R)$ for any $R>0$. Suppose $u_1$ is such a solution on $B(0,2R)$. Then $\eta_1\colonequals u_1\cdot f_1$ is holomorphic on $B(0,R)$ and the zero locus $Z(f)=Z(\eta_1)$ is discrete on $B(0,R)$. Since $R$ is arbitrary, $Z(f_1)$ is discrete. By Lemma \ref{finitezeros}, $Z(f_1)$ is also finite, so it lies in a compact region of $\C$. Set
\[
f_2=\prod_{z_i\in Z(f)} (z-z_i). 
\] 

The function $u\colonequals f_1/f_2$ is non-vanishing on $\C$ and $\bpartial_\omega u=0$. Since $\C$ is simply connected, $u=e^\zeta$ for some smooth $\zeta: \C\to \C$. Then
\[
\bpartial u+\omega^{0,1} u=0\Rightarrow\bpartial(-\zeta)=\omega^{0,1}. 
\]

This shows $\zeta$ is a global solution to the equation (\ref{holo}). Since on each circle $\partial B(0,r_n)$, $|f_1|>c>0$, we can find $C>0$ such that
\[
|e^{-\zeta(z)}|=|f_2/f_1|<C|z|^d
\]
for any $n$ and $z\in \partial B(0,r_n)$.

Now consider $(d,\eta)\colonequals e^{-\zeta}(\nabla_\omega, f')$. Then $\bpartial\eta=0$. By Lemma \ref{boundedness2}, $f_1\in L^\infty(\C)$. Thus, 
\begin{equation}\label{polygrowth}
|\eta(z)|<C_1|z|^d
\end{equation}
for any $n$ and $z\in\partial B(0,r_n)$. Apply maximal principle to $\eta/z^d$ on the annulus $B(0, r_{n+1})\backslash B(0, r_n)$. We conclude that $\eta/z^d$ is uniformly bounded when $|z|>r_1$. Hence, $\eta/z^d$ extends to $\infty$ and $\eta$ is a polynomial map. This completes the proof of Theorem \ref{polynomial}. \qedhere
\end{proof}

Now we turn to the proof of Lemma \ref{finitezeros}. We start with a lemma that generalizes the classical theorem in complex analysis:
\begin{lemma}\label{complexanalysis}
Suppose $d+\omega$ is a unitary connection on the trivial line bundle on $\C$ and a complex-value function $f:\C\to \C$ is a holomorphic with respect to $\omega$, i.e, $\bpartial_\omega f=(df+\omega\otimes f)^{0,1}=0$. If $f$ is non-vanishing on $|z|=R$, then 
\[
\#\{z\in B(0,R):f(z)=0\}=\frac{1}{2\pi i}\int_{|z|=R}\frac{\nabla_\omega f}{f}-\omega.
\]
\end{lemma}
\begin{proof} The $\bpartial$-Poincar\'{e} lemma allows us to find a complex gauge transformation $u$ such that $\eta=u\cdot f$ is holomorphic on $B(2R,0)$. Then $\eta$ and $f$ have the same zero locus $Z(\eta)=Z(f)$. Because the 1-form $udu^{-1}$ is closed on $B(0,R)$, we have, 
\begin{align*}
\frac{1}{2\pi i}\int_{|z|=R}\frac{\nabla_\omega f}{f}-w&=\frac{1}{2\pi i}\int_{|z|=R} \frac{df}{f}=\frac{1}{2\pi i}\int_{|z|=R}\frac{d\eta}{\eta}+udu^{-1}\\
&= \frac{1}{2\pi i}\int_{|z|=R}\frac{\partial\eta}{\eta}=\# Z(f)\cap B(0,R).
\end{align*}
\end{proof}

From now on, we borrow ideas from Wehrheim's paper \cite{Wehrheim}. We use the polar coordinate $(r,\theta)$ on $\C$. Write $\nabla_A$ as 
\[
\nabla_A =\nabla_B+(\frac{\partial}{\partial r}+h)\otimes dr+(\frac{1}{r}\frac{\partial}{\partial\theta}+g)\otimes rd\theta. 
\]
We compute the curvature form:
\[
F_A=F_Bdvol_\Sigma+F_A^\C dvol_\C+dr\wedge (\frac{\partial B}{\partial r}-d_\Sigma h)+rd\theta\wedge (\frac{1}{r}\frac{\partial B}{\partial \theta}-d_\Sigma g). 
\]

If we regard $(A,\sigma)$ as configuration on $\R\times S^1\times \Sigma$ and ignore the $dr$ component of $\nabla_A$, we get a family of configurations on $Y=S^1\times \Sigma$. Let us denote them by $(A_r,\sigma_r)$. Then 
\[
\nabla_{A_r}^\Sigma|_{(\theta,x)}=\nabla_{B(r,\theta,x)},\ \nabla_{A_r}^{\theta}|_{(\theta,x)}=\frac{\partial}{\partial\theta}+rgd\theta,\ \sigma_r(\theta,x)=\sigma(r,\theta,x).
\]

If we decompose $F_{A_r}$ into its $\Sigma$-part and its mixed part, we obtain
\[
F_{A_r}^\Sigma(\theta)=F_{B(r,\theta,x)},\ F^1_{A_r}=d\theta\wedge (\frac{\partial B}{\partial \theta}-rd_\Sigma g).
\]

Note there are two different metrics on $\R\times Y$. One is the product metric, the other is induced from polar coordinates. Whenever the symbol $Y$ is used, we indicate the first metric, while the second is used implicitly for $\partial X_r$. 
For any $r>0$, define $T(r)$ by the formula:
\begin{align*}
\int_{\partial X_r} r|F^m_A|^2+r|\nabla_A^\C\sigma|^2+2|\bpartial_A^\Sigma\sigma|^2+|iF_{A}^\Sigma+\half K+\half|\sigma|^2|^2.
\end{align*}

Note that $T(r)$ controls the analytic energy of $(A_r,\sigma_r)$ on $Y$ when $r>1$. Indeed, 
\begin{align*}
\E_{an}(A_r,\sigma_r):&= \int_Y |F_{A_r}^1|^2+|\nabla_{A_r}^\theta \sigma|^2+2|\bpartial_{A_r}^\Sigma\sigma_r|^2+|iF_{A_r}^\Sigma+\half K+\half|\sigma_r|^2|^2\\
&\leq \int_{\partial X_r} r|F^m_A|^2+r|\nabla_A^\C\sigma|^2+\frac{2}{r}|\bpartial_A^\Sigma\sigma|^2+\frac{1}{r}|iF_{A}^\Sigma+\half K+\half|\sigma|^2|^2\\
&\leq T(r). 
\end{align*}

In addition,  for $r>1$, 
\begin{align*}
\E_{an}(A,\sigma)&=\int_{0}^\infty \frac{d}{dr}\E(r)\geq \int_1^\infty \frac{1}{r}T(r).
\end{align*}
This implies 
\begin{lemma}
There exists a sequence of numbers $r_n>1$ such that $\lim_{n\to\infty} r_n=\infty$ and $\lim_{n\to\infty}T(r_n)= 0$.
\end{lemma}

Recall that $F_A'= F_A-i\frac{K}{2}dvol_\Sigma$. Let $a=A-A_0$. Then $F_A'=F_{A_0}'+da$. Since $\nabla_{A,\vec{n}}\sigma=-i/r\cdot \nabla_{A,\partial_\theta}\sigma$ and $F_{A_0}'$ contains only $dvol_\Sigma$ component, by Lemma \ref{energyball}: 
	\begin{align*}
\E(r)&=-\frac{1}{r}\int_{\partial X_r} \langle \sigma, i \nabla_{A,\partial_\theta}\sigma\rangle+ \int_{X_r} F_A'\wedge F_A'-\int_{X_r} F_{A_0}'\wedge F_{A_0}'\\
&=\int_Y- \langle \sigma_r, i \nabla_{A_r}^\theta\sigma_r\rangle+a\wedge (2F_{A_0}'+da). 
\end{align*}

Note that 
\[
a=-i*_\C d_\C\alpha -i*_\Sigma d_\Sigma\alpha +\omega.
\]

Let $a_r=a|_{\partial X_r}, \omega_r=\omega|_{\partial B(0,r)}$ and $\bar{\omega}(r)=\int_{S^1}\omega_r$. For each configuration $(A_r,\sigma_r)$ on $Y$, we apply the gauge fixing condition:
\begin{equation}\label{fixing}
\omega_r=\bar{\omega}(r)d\theta.
\end{equation}

This can be achieved since $\omega_r-\bar{\omega}(r)=d\beta_r$ for some $\beta_r\in \Gamma(S^1,i\R)$ and we can work with $e^{\beta_r}\cdot (A_r,\sigma_r)$ instead. Note that (\ref{GFC}) and (\ref{fixing}) are different from the Coulomb gauge fixing condition on $Y$. In terms of the Hodge decomposition of $\Omega^1(Y,i\R)$, write
\[
a_r-\omega_r=a_r^1+a_r^h+a_r^2,
\]
where $a_r^1, a_r^h$ and $a_r^2$ are exact, harmonic and co-exact parts of $a_r-\omega_r$ respectively. Since pull-backs from $S^1$ or $\Sigma$ generate the space of harmonic 1-forms on $Y$, $a_r^h=0$. The exact component $a_r^1$ is nonzero in general. By the gauge fixing condition (\ref{fixing}), we have 
\[
\int_Ya\wedge (2F_{A_0}'+da)=-4\pi^2 ic_1(S^+)\bar{\omega}(r)+H(a_r^2).
\]
 Here, $H(a_r^2)$ is a function that involves $\alpha$ only. Indeed, 
 \[
 H(a_r^2)=\int_{X_r} (-i*_\C d_\C\alpha)\wedge (2F_{A_0}'+2i\Delta_\Sigma\alpha dvol_\Sigma).
 \]
It depends only on the co-exact component $a_r^2$ and is continuous with respect to $L^2_{1/2}(Y,i\Omega^1(Y))$-topology.
 
Since $\E_{an}(A_{r_n},\sigma_{r_n})\leq T(r_n)\to 0$, we have 
\[
\|(a_{r_n}^2,\sigma_{r_n})\|_{L^1_2(Y)}<C
\] 
for some uniform $C>0$. See \cite[Theorem 5.5.1]{Bible} for a proof for 4-dimensional equations. By passing to a subsequence, we may assume $(a_{r_n}^2,\sigma_{r_n})$ converge weakly in $L_1^2$. Therefore, the sequence
\[
\E(r_n)+4\pi^2 ic_1(S)\bar{\omega}(r_n)=H(a_{r_n}^2)-\int_Y \langle \sigma_r, i \nabla_{A_r}^\theta\sigma_r\rangle
\]
converges. Since $\E(r)$ has a limit as $r\to\infty$, $\lim \bar{\omega}(r_n)$ exists. This implies that for some proper gauge transformations $e^{\beta_n}$ and an $L_1^2$-configuration $(A_\infty, \sigma_\infty)$ on $Y$, 
\begin{equation}\label{L21}
e^{\beta_n}(A_{r_n},\sigma_{r_n})\xrightarrow{w-L_1^2} (A_\infty, \sigma_\infty),
\end{equation}
and $\E_{an}(A_\infty, \sigma_\infty)=0$. So, for some $m\in \Z$, 
\begin{equation}\label{limit}
A_\infty=\frac{\partial}{\partial\theta}+B_0'-imd\theta, \sigma_\infty=e^{im\theta}\cdot \gamma,
\end{equation}
and $(B_0', \gamma)$ is a vortex on $\Sigma$, i.e. this pair solves the  vortex equation (\ref{vortex}). 

At this moment, we do not know $e^{\beta_n}\sigma_{r_n}\to \sigma_\infty$ in $L^\infty$-norm since in dimension $3$, $L^2_1\not\embed L^\infty$. We only need a weaker result and it is almost there. We examine the exact part of $a_r-\omega_r$ more carefully:
\begin{lemma}\label{exact}
$\|a_{r_n}^1\|_{L^2(Y)}\to 0$ as $r_n\to\infty$. 
\end{lemma}
\begin{proof}
The exact part of $a_r-\omega_r$ arises form $-i*_\C d_\C\alpha|_{\partial X_r}$. Indeed, 
\[
\delta_Y (-i*_\Sigma d_\Sigma\alpha)=0,
\]
where $\delta_Y$ is the formal adjoint of the exterior differential $d_Y$. Therefore, 
\begin{align*}
\int_Y |a_r^1|^2 &\leq r\int_{\partial X_r}| -i*_\C d_\C\alpha|^2\leq r\lambda_1^{-1}\int_{\partial X_r}|d_\C d_\Sigma\alpha|^2,
\end{align*} 
using the fact that $\alpha(z)\in (\ker \Delta_\Sigma)^\perp$. $\lambda_1$ is the first positive eigenvalue of $ \Delta_\Sigma$. On the other hand, since $F^m_A=i(*_\C-*_\Sigma)d_\C d_\Sigma \alpha$,
\[
\frac{1}{r}T(r)\geq\int_{\partial X_r}|F^m_A|^2=2\int_{\partial X_r}|d_\C d_\Sigma\alpha|^2. 
\]
The last equality follows from the fact that over each fiber $\{z\}\times\Sigma$, $*_\C d_\C d_\Sigma\alpha$ is an exact form while $*_\Sigma d_\C d_\Sigma\alpha$ is co-exact, so they are orthogonal. 

Finally, $T(r_n)\to 0$ implies $\|a^1_{r_n}\|_{L^2(Y)}\to 0$. 
\end{proof}

Now we are ready to prove Lemma $\ref{finitezeros}$. 
\begin{proof}[Proof of Lemma \ref{finitezeros}] We take $\beta_n\in (\ker \Delta_Y)^\perp$ such that $d_Y\beta_n=a_r^1$. By Lemma \ref{exact}, $\|\beta_n\|_{L^2_1(Y)}\to 0$ as $n\to\infty$. This shows that gauge fixing conditions $(\ref{GFC})$ and $(\ref{fixing})$ are satisfied for the limit $(A_\infty,\sigma_\infty)$, so $(\ref{limit})$ holds without further gauge transformations. Let 
\[
l_n(\theta)=\int_{\{\theta\}\times\Sigma} \langle e^{\beta_n}\sigma_{r_n}, \sigma_\infty\rangle,\  l_n'(\theta)=\int_{\{\theta\}\times\Sigma} \langle \sigma_{r_n}, \sigma_\infty\rangle,
\]
then $l_n, l_n'\in L_1^2(S^1)$. Moreover, $l_n(\theta)$ converges to the constant function $\|\gamma\|_{L^2(\Sigma)}^2$ in $L^\infty(S^1)$-topology. Since $\beta_n$ is imaginary, $|e^{\beta_n}-1|\leq C|\beta_n|$ for some $C>0$. Thus, 
\begin{equation}\label{310}
\|l_n-l_n'\|_{L^{3/2}_1(S^1)}\leq C_1\|\beta_n\|_{L^2_1(Y)}\to 0.
\end{equation}

Indeed, by Lemma \ref{boundedness2}, $\sigma\in L^\infty(X)$ and $\sigma_\infty\in L^\infty(Y)$, so 
\[
|l_n(\theta)-l_n'(\theta)|=|\int_{\{\theta\}\times\Sigma} \langle (e^{\beta_n}-1)\sigma_{r_n}, \sigma_\infty\rangle|\leq C\int_{\{\theta\}\times \Sigma}|\beta_n|.
\]

To deal with the derivative, note that  $\frac{d}{d\theta}(l_n-l_n')$ is bounded by
\[
 |\int_{\{\theta\}\times\Sigma} \langle \frac{d\beta_n}{d\theta}e^{\beta_n}\sigma_{r_n}, \sigma_\infty\rangle|+|\int_{\{\theta\}\times\Sigma} \langle (e^{-\beta_n}-1)\frac{d}{d\theta}(e^{\beta_n}\sigma_{r_n}), \sigma_\infty\rangle|.
\]

The first term is controlled in the same way. For the second, we use the multiplicative structure $L^6\times L^2\embed L^{3/2}$ and Sobolev embedding theorem $L^2_1\embed L^6$ in dimension $3$. This proves estimate (\ref{310}).

By the gauge fixing condition (\ref{GFC}), $B_0'-B_0$ is a co-exact 1-form on $\Sigma$. Then $B_0'-B_0=-i*_\Sigma d_\Sigma\alpha_\infty$ for a unique function $\alpha_\infty\in (\ker\Delta)^\perp$. Let $v=e^{\alpha_\infty}\sigma_\infty$. Then $\bpartial_{B_0}v=0$. Recall that $\sigma_r'=e^\alpha\sigma_r$ and for any $z\in \C$, 
\[f_1(z)=\int_{\{z\}\times\Sigma} \langle \sigma_r',v\rangle.\]
By Lemma \ref{boundedness2}, $\alpha\in L^\infty(X)$. Since $\alpha_\infty\in L^\infty(\Sigma)$, for any $z=r_ne^{i\theta}\in \C$, 
\[
|f_1(z)|\geq c_2 |l'_n(\theta)|.
\]

This implies that when $n\gg 0$, $|f_1(r_ne^{i\theta})|>c$ for some $c>0$, since the same holds for $l_n(\theta)$ and $l_n'(\theta)$. 
\medskip

Finally, we need to verify that $f_1$ has finitely many zeros. We apply Lemma \ref{complexanalysis} and give an upper bound for that integral. The contribution from the connection form is settled since it is just $\bar{\omega}_{r_n}$ and $\lim \bar{\omega}_{r_n}$ exists. Since $|f_1|>c$, for $r=r_n$,
\[
|\int_{\partial B(0,r)} \nabla_\omega f_1/f_1|\leq \frac{1}{c} \int_{S^1}|\nabla_{\omega_r}^\theta f_1|\leq \frac{C}{c} \|\nabla_{\omega_r}^\theta\sigma_r'\|_{L^2(Y)}.
\]
It is sufficient to estimate $\|\nabla_{\omega_r}\sigma_r'\|_2$. By (\ref{covariantderivative}), 
\[
\nabla_\omega \sigma'=e^{\alpha}(\nabla_{A}^\C\sigma+2(d_\C\alpha)^{1,0}\otimes \sigma).
\]
By Lemma \ref{boundedness2}, $\alpha, \sigma\in L^\infty(X)$. It suffices to estimate the $L^2$-norms of $\nabla_A^\C\sigma$ and $d_\C\alpha$. For the first term,
\[
\int_{Y}|\nabla_{A_r}^\theta\sigma_r|^2\leq \int_{\partial X_r}r|\nabla_A^\C\sigma|^2\leq T(r).
\]
For the second, it was done in proof of Lemma \ref{exact}. This completes the proof of Lemma \ref{finitezeros}
\end{proof}

\section{When $c_1(S^+)=0$}\label{4.0}

In this section, we discuss the case when $c_1(S^+)=0$ and prove Theorem \ref{flat}. In this case, finite energy monopoles are necessarily reducible and they are identified with the moduli space of flat connections on $\bigwedge^2 S^+$. We reformulate the result in terms of the vortex equation:
\begin{theorem}\label{reducible}
Any finite energy solution $(A,\sigma)$ to the equation  $(\ref{truncatedmonopole})$ is reducible, i.e., $\sigma\equiv 0$ on $X$. In addition, $\hat{A}$, the induced connection on $\bigwedge^2 S^+$, is flat. 
\end{theorem}

\begin{proof}
We shall use notations from the last section. Since $c_1(S^+)=0$, we can choose a background connection $B_0$ on $\Sigma$ such that $iF_{B_0}+\half K\equiv 0$ and, after imposing the gauge fixing condition (\ref{GFC}), the connection $A$ is given by
\[
\nabla_A =\nabla_0+\frac{\partial}{\partial u}+\frac{\partial}{\partial v}-i*_\C d_\C\alpha-i*_\Sigma d_\Sigma \alpha+\omega. 
\]
for some smooth function $\alpha\in \SC^\infty(X)$ with $\int_{\Sigma}\alpha(z,\cdot)=0$ on each fiber. Here, $\omega\in \Gamma(\C, iT^*\C)$ is an imaginary $1$-form. Therefore, $F^\C_A=F_\omega+i\Delta_{\C}\alpha$. Integrating equation (\ref{firsteq}) over each fiber, we obtain
\begin{align*}
0=\int_\Sigma iF_A^\C +\int_\Sigma (iF_A^\Sigma+\half K)+\half\int_{\Sigma}|\sigma|^2&=iVol(\Sigma) F_\omega+\pi c_1(S^+)+\half\int_{\Sigma}|\sigma|^2\\
&= iVol(\Sigma) F_\omega+\half\int_{\Sigma}|\sigma|^2. 
\end{align*}

This shows $iF_\omega\leq 0$. Let $(A',\sigma')=e^{\alpha}\cdot (A,\sigma)$ and set $T(z)\colonequals \int_{\{z\}\times\Sigma}|\sigma'|^2$. By the proof of Lemma \ref{boundedness}, we have
 \[T(z), \alpha\to 0\]
 as $z\to\infty$. Indeed, for a solution $(A,\sigma)$ on $X_5=B(0,5)\times \Sigma$ of the equation (\ref{truncatedmonopole}) with zero analytic energy, we necessarily have $\sigma\equiv 0$ and $\alpha\equiv 0$ since $c_1(S^+)=0$. In particular, $T\in L^\infty(\C)$ is a bounded function. 

\bigskip
Since $\nabla_{A'}^\C=\nabla_\omega$ and $\bpartial_\omega \sigma'=0$, by Weitzenb\"{o}ck formula $(\ref{overC})$,
\begin{align*}
\Delta_{\C} T&=-2\int_{\Sigma}|\nabla_\omega\sigma'|^2+\langle \nabla_\omega^*\nabla_\omega \sigma',\sigma'\rangle\leq 2\int_\Sigma \langle iF_\omega\sigma',\sigma'\rangle\leq 0. 
\end{align*}

Therefore, $T$ is a bounded subharmonic function on $\C$, so $T$ is constant. Because $\lim_{z\to\infty} T(z)=0$, $T\equiv 0$. It follows that $F_\omega\equiv 0$.

Finally, equation (\ref{firsteq}) shows that $-\Delta_X \alpha\equiv 0$, so $\alpha$ cannot attain its maximum or minimum in the interior of any bounded domain. But $\alpha\to 0$ as $z\to\infty$. Thus, $\alpha\equiv 0$ and $\hat{A}$ is flat.
\end{proof}

\section{Polynomials $\Rightarrow$ Vortices}\label{4}
So far, we have not seen any smooth solution to equation (\ref{SWEQ}) or (\ref{truncatedmonopole}) on $X=\C\times\Sigma$ that has nonzero energy. In Section \ref{51}, we take up the task of constructing solutions. Starting with a polynomial map $f:\C\to H^0(\Sigma,\SL)$, we produce a vortex $(A,\sigma)$ such that $Z(\sigma)=Z(f)$. This solution exists, a priori, in the Fr\'{e}chet space $L_{2,loc}^2(X)$, but we will show it is smooth and unique in Section \ref{52}. There is a tedious a priori estimate that appears in variational principle and we postpone its proof to the next section. 

\subsection{Existence of solutions}\label{51}
Let us recall some setup from the previous section. Let $\SL=(L,\bpartial_\SL)$ be a holomorphic structure on $L\to \Sigma$ and let $f:\C\to H^0(\Sigma,\SL)$ be a nonzero polynomial map of degree $d$. This means there are some global sections $\gamma_i\in H^0(\Sigma,\SL), 0\leq i\leq d$ with $\gamma_d\neq 0$ such that
\[
f(z)=\sum_{i=0}^{d}\gamma_i z^i
\]
for any $z\in\C$. Suppose a Hermitian metric $h$ on $L$ is fixed. The Chern connection on $\SL$ is the unique unitary connection $B_0=\nabla_0$ such that
\[
\nabla^{0,1}_0=\bpartial_{\SL}. 
\]

We impose an extra condition on the pair $(B_0=\nabla_0,\gamma_d)$: this configuration solves the vortex equation (\ref{vortex}) on $\Sigma$:
\begin{equation}\label{vortex2}
\left\{\begin{array}{r}
*iF_B+\half K+\half |\sigma|^2=0,\\
\bpartial_B\sigma=0.
\end{array}
\right.
\end{equation}
This can be achieved by applying an element in $\SG_\C(\Sigma)$ since the solvability constraint 
\[
0>\pi c_1(S^+)=\int_\Sigma iF_B+\half Kdvol_\C
\]
is satisfied. For a proof, see \cite[Theorem 4.3]{Bradlow}, \cite[Theorem]{Garcia} or Theorem \ref{B1}. 

The line bundle $\SL$ pulls back to a holomorphic line bundle on $X=\C\times\Sigma$ and $f$ is regarded as a section on $X$ by setting $\sigma_0(z,x)=f(z)(x)$. The connection $B_0$ induces on $X$ a unitary connection:
\[
\nabla_{A_0}=\nabla_0+\frac{\partial}{\partial u}+\frac{\partial}{\partial v}.
\]

The conformal transformation is defined on the configuration space $\SC(X,L)$ by the formula: 
\begin{align}
g=e^\alpha: (A,\sigma)&\mapsto( A+i*_\C d_\C\alpha+i*_\Sigma d_\Sigma\alpha,e^{\alpha}\sigma).
\end{align}

The curvature and covariant derivative are transformed accordingly:
\begin{align}
g=e^\alpha: F_A&\mapsto F_A-i\Delta_\C\alpha dvol_\C-i\Delta_\Sigma\alpha dvol_\Sigma+F^-(\alpha)\\
\nabla_A\sigma&\mapsto e^{\alpha}(\nabla_A\sigma+2(d\alpha)^{1,0}\otimes \sigma), \nonumber
\end{align}
where $F^-(\alpha)$ reflects the change of the mixed term (the $F^m_A$-part) and it lies in $\Lambda^-(X)\subset \Lambda^2(X)$. Note that for $(A,\sigma)=e^\alpha\cdot (A_0,\sigma_0)$, equations (\ref{secondeq}) and (\ref{thirdeq}) are automatically satisfied.
\begin{theorem}\label{existence}
For any polynomial map $f$ of degree d, we can find $\talpha\in \SC^\infty(X)$ such that $(A,\sigma)=e^\talpha\cdot (A_0, \sigma_0)$ solves the equation $(\ref{firsteq})$:
\begin{equation}\label{first}
i(F_A^\Sigma+F_A^\C)+\half K+\half |\sigma|^2=0,
\end{equation}
and its analytic energy $\E_{an}(A,\sigma)$ equals $-4\pi^2d\cdot c_1(S)$. In particular, $(A,\sigma)$ gives a finite energy monopole on $X$. 
\end{theorem}

\Remark For $(A,\sigma)=(A_0,\gamma_d)$,  that is, we extend $\gamma_d$ to be constant in variable $z\in \C$, this pair solves equation (\ref{truncatedmonopole}) since $(B_0,\gamma_d)$ solves (\ref{vortex2}). This corresponds to the case when $d=0$ in Theorem \ref{existence}.  

\bigskip
Before we start the actual proof, let's sketch a strategy to find such an $\talpha$:
\smallskip

\textit{Step 1.} Choose a background conformal transformation $\alpha_0\in \SC^\infty(X)$ and set
\[
(A_1,\sigma_1)=e^{\alpha_0}\cdot (A_0,\sigma_0).
\]

At this step, the configuration $(A_1,\sigma_1)$ is not necessarily a solution to (\ref{first}). It is close to an actual solution so that the analytic energy $\E_{an}(A_1,\sigma_1)$ is finite. Moreover, the next step needs to be achieved:

\smallskip
\textit{Step 2.} We find another conformal factor $\alpha\in L^2_2(X)$ such that $(A_\alpha,\sigma_\alpha)\colonequals e^\alpha\cdot (A_1,\sigma_1)$ solves (\ref{first}). Take $\talpha=\alpha_0+\alpha$. 
\begin{definition}\label{momentmap}
	We define the moment map $\mu$ as 
	\begin{align*}
	\mu: L^2_2(X)&\to L^2(X),\\
	\alpha&\mapsto i(F_{A_\alpha}^\Sigma+F_{A_\alpha}^\C)+\half K+\half |\sigma_\alpha|^2\\
	&=\mu(0)+(\Delta_\C\alpha+\Delta_\Sigma\alpha)+\half (e^{2\alpha}-1)|\sigma_1|^2.
	\end{align*}
\end{definition}

The second step amounts to finding $\alpha\in L^2_2(X)$ so that $\mu(\alpha)=0$. The definition of $\mu$ depends on $\alpha_0$. We wish $\mu$ to be well-defined so that we may apply variational principle to $\|\mu(\alpha)\|_2^2$. Our target $\alpha$ would be the minimizer of this functional. The first guess for $\alpha_0$ is 
\[
\alpha_0=-\frac{d}{2}\log(|z|^2+1).
\]
But in general, this choice does not guarantee that $\mu$ is a well-defined map from $L^2_2(X)$ to $L^2(X)$. 
\begin{lemma}\label{53} We can find $\alpha_0\in \SC^\infty(X)$ so that
	for any $\alpha\in L^2_2(X)$, $\mu(\alpha)$ is square-integrable and the energy $\E(\alpha)\colonequals\E_{an}(A_\alpha,\sigma_\alpha)$ is finite. Furthermore,  for this $\alpha_0$ and $\mu$, the energy equation
	\begin{align}\label{Energy2}
	\int_{X} |\mu(\alpha)|^2=\E_{an}(A_\alpha,\sigma_\alpha)-\E_{top} 
	\end{align}
	is valid. The topological energy is defined by the formula
	 \[\E_{top}=-4\pi^2d\cdot c_1(S^+),\]
	 which depends only the degree of $f$ and $c_1(S^+)$.  
\end{lemma}

\begin{proof}
Write $\alpha_0=\beta+\delta$ with 
\[
\beta=-\frac{d}{2}\log(|z|^2+1)
\]
and $\delta$ to be determined later. Using the fact that $(B_0,\gamma_d)$ solves the vortex equation (\ref{vortex2}), we have  
\begin{align*}
\mu(0)&=(F_{A_1}^\Sigma+F_{A_1}^\C)+\half K+\half |\sigma_1|^2\\
&=\Delta_\C(\beta+\delta)+\Delta_\Sigma\delta+*_\Sigma iF_{B_0}+\half K+\half |\sigma_1|^2\\
&=\Delta_\C(\beta+\delta)+\Delta_\Sigma\delta+\half (|\sigma_1|^2-|\gamma_d|^2).
\end{align*}

We attempt to make  $\mu(0)$ in $L^2(X)$.

If $\delta=0$, then $\sigma_1=\sigma_*\colonequals f(z)/(1+|z|^2)^\frac{d}{2}\in L^\infty(X)$ and by direct computation:
\begin{align}
\Delta_\C\beta&=\frac{2d}{(1+r^2)^2}\in L^2(\C)\label{Laplacian},\\
|\sigma_*|^2-|\gamma_d|^2&=\frac{2|z|^{2d-2}\re\langle z\gamma_d,\gamma_{d-1}\rangle_h}{(1+|z|^2)^d}+\SO(\frac{1}{1+|z|^2}). \nonumber
\end{align}

If we know $\gamma_{d-1}\equiv 0$, then $|\sigma_1|^2-|\gamma_d|\in L^2(X)$ and we are done. 
To deal with the general case, note that the unbounded operator
\[
T=\Delta_\Sigma+|\gamma_d|^2: L^2(\Sigma)\to L^2(\Sigma)
\]
is self-adjoint on $L_2^2(\Sigma)$ and is invertible. Set
\[
\delta(z)=-\frac{|z|^{2d-2}}{(1+|z|^2)^d}T^{-1}(\re\langle z\gamma_d,\gamma_{d-1}\rangle_h)\in \SO(\frac{1}{\sqrt{1+|z|^2}}).
\]

Then we have,
\begin{align*}
\Delta_\Sigma\delta+\half (|\sigma_1|^2-|\gamma_d|^2)&=\Delta_\Sigma\delta+\half |\sigma_*|^2(e^{2\delta}-1)+\half(|\sigma_*|^2-|\gamma_d|^2),\\
&=(|\sigma_*|^2-|\gamma_d|^2)\delta+\half |\sigma_*|^2(e^{2\delta}-2\delta-1)\\
&\qquad+\half(|\sigma_*|^2-|\gamma_d|^2+2T(\delta))\in \SO(\frac{1}{1+|z|^2}).
\end{align*}

To check $\Delta_\C\delta\in L^2(X)$, it suffices to compute:
\[
\Delta_\C \frac{|z|^{2d-2}u}{(1+|z|^2)^d}=\Delta_\C\frac{r^{2d-1}\cos\theta}{(1+r^2)^d}=\frac{4dr^{2d-3}(2r^2-(d-1))\cos\theta}{(1+r^2)^{d+2}},
\]
where $r=|z|$ and $u=r\cos\theta$. This function has enough decay at $\infty$ and lies in $L^2(X)$. 

To show $\mu(\alpha)\in L^2(X)$ in general, it suffices to check:
\[
\Delta_X\alpha,\  \half(e^{2\alpha}-1)|\sigma_1|^2\in L^2(X).
\]

The first follows from the fact that $\alpha\in L_2^2(X)$. The second comes from  Trudinger's inequality (Theorem \ref{A2}) and the fact that $\sigma_1\in L^\infty(X)$. 

It remains to prove (\ref{Energy2}): it will imply that analytic energy $\E_{an}(\alpha)$ is finite. We first do the case when $\alpha\in \SC_c^\infty(X)$. In light of Lemma \ref{energyball}, it suffices to show
\begin{align}
\lim_{r\to\infty}\int_{\partial X_r} \langle \sigma_\alpha, \nabla_{A_\alpha,\vec{n}}\sigma_\alpha\rangle&= 0,\label{boundary}\\
 \lim_{r\to\infty}\int_{X_r} F_{A_\alpha}'\wedge F_{A_\alpha}'&=-4\pi^2 d\cdot c_1(S^+)\label{topo} .
\end{align}

Suppose $\supp(\alpha)\subset B(0,r_0)$ for some $r_0>0$ and take $r>r_0$. Then $(A_\alpha,\sigma_\alpha)=(A_1,\sigma_1)$.  Let $(A_*, \sigma_*)=e^\beta\cdot (A_0,\sigma_0)$. By formula (\ref{covariantderivative}), 
\begin{align*}
\nabla_{A_\alpha,\vec{n}}\sigma_\alpha &= \nabla_{A_1}\sigma_1=e^{\delta}(\nabla_{A_*}\sigma_*+2(d\delta)^{1,0}\sigma_*).
\end{align*}

Since $\sigma_1, \sigma_*, \delta\in L^\infty(X)$ and 
\begin{align*}
\nabla_{A_*}\sigma_*&=e^{\beta}(d\sigma_0+2(d\beta)^{1,0}\sigma_0),\\
&=\frac{-\gamma_{d-1}\cdot r^2z^{d-2}+\text{lower order terms}}{(1+|z|^2)^{(d+2)/2}}\cdot dz\in \SO(\frac{1}{1+r^2})\\
d( \frac{r^{2d-1}\cos\theta}{(1+r^2)^d})& \in \SO(\frac{1}{1+r^2}),
\end{align*}
the boundary term goes to zero in $(\ref{boundary})$ as $r\to\infty$.

\bigskip
To compute the topological energy $\E_{top}$,  by formula (\ref{Laplacian}), we have
\begin{align*}
\int_{X_r} F_{A_*}'\wedge F_{A_*}'&=-2\int_{X_r}\Delta_{\C}\beta\cdot (iF_{B_0}^\Sigma+\half K)dvol_X\\
&=-2\pi c_1(S^+)\int_{B(0,r)} \Delta_{\C}\beta \to -4\pi^2 d\cdot c_1(S^+)
\end{align*}
as $r\to\infty$. Let $\mu=A_\alpha-A_*$. Then
\begin{align*}
\int_{X_r}( F_{A_\alpha}'\wedge F_{A_\alpha}'-F_{A_*}'\wedge F_{A_*}')&=\int_{\partial X_r} \mu\wedge (d\mu+2F'_{A_*}). 
\end{align*}

When $r>r_0$, $\alpha\equiv 0$. We reduce to the case when $A_\alpha=A_1$. Since $|\mu|<|d\delta|\sim 1/(1+r^2)$ and the curvature term $d\mu+2F_{A_*}'$ is uniformly bounded on $X$, the integral above decays as $1/r$ as $r\to\infty$. Therefore, formula (\ref{topo}) is also valid.

We showed that the energy equation (\ref{Energy2}) holds for any $\alpha\in\SC_c^\infty(X)$. It also holds for any $\alpha\in L^2_2(X)$ since $\SC_c^\infty(X)$ is dense in $L^2_2(X)$ and all terms in (\ref{analyticenergy}), as functions in $\alpha$, are continuous in $L^2_2(X)$-topology. We may need Theorem \ref{A3} to verify the continuity. \qedhere

\end{proof}

The next theorem is an a priori estimate and the proof is technical. Its proof is postponed to the next section. We will finish the proof of Theorem \ref{existence} assuming Theorem \ref{apriori3}.

\begin{theorem}\label{apriori3}
	For any $\alpha\in L^2_2(X)$, define $\E(\alpha)=\E_{an}(A_\alpha,\sigma_\alpha)$. There is a function $\eta: \R^+\to \R^+$ such that for any $C>0$, if  $\E(\alpha)<C$, then $\|\alpha\|_{L^2_2}<\eta(C)$. 
\end{theorem}

\begin{proof}[Proof of Theorem \ref{existence}]
Let $a=\inf_{\alpha\in L^2_2(X)}\E(\alpha)$. This number is finite since $\E(0)<\infty$. Therefore, there exists a sequence $\{\alpha_n\}\subset L_2^2(X)$ such that $$a=\liminf_{n\to\infty} \E(\alpha_n).$$ By Theorem \ref{apriori3}, $\|\alpha_n\|_{L^2_2(X)}$ are uniformly bounded, and we can find a weakly convergent subsequence. We assume it is the sequence itself. Let $\alpha_\infty$ be their limit. Note that $\mu: L^2_2(X)\to L^2(X)$ is weakly continuous. Indeed, by Theorem \ref{A4},  $\alpha_n\xrightarrow{w-L_2^2}\alpha_\infty$ implies 
\[
\Delta_X\alpha_n\xrightarrow{w-L^2(X)}\Delta_X\alpha_\infty,\  e^{\alpha_n}-1\xrightarrow{w-L^2(X)}e^\alpha-1. 
\] 

This shows $\E(\alpha_\infty)\leq \liminf \E(\alpha_n)=a$, so $\E(\alpha_\infty)=a$. Now consider the linearized operator of $\mu$ at $\alpha_\infty$ (see Theorem \ref{A3}):
\[
\mathcal{D}_{\alpha_\infty}\mu: L^2_2(X)\to L^2(X), \gamma\mapsto \Delta_X\gamma+e^{2\alpha_\infty}|\sigma_1|^2\gamma.
\]

Since $\alpha_\infty$ is a critical point of $\E$, for any $\gamma\in L^2_2(X)$, we have
\begin{equation}\label{critical}
0=\frac{d}{dt}\E(\alpha_\infty+t\gamma)|_{t=0}=2\langle \mu(\alpha_\infty), \mathcal{D}_{\alpha_\infty}\mu(\gamma)\rangle.
\end{equation}

\begin{lemma}\label{SA}
For any $\alpha\in L^2_2(X)$, the operator 
\[
\mathcal{D}_\alpha\mu: L^2_2(X)\to L^2(X), \gamma\mapsto \Delta_X\gamma+e^{2\alpha}|\sigma_1|^2\gamma,
\]
is self-adjoint. 
\end{lemma}
\begin{proof}
The operator $\mathcal{D}_\alpha\mu$ is well-defined. Indeed, by Theorem \ref{A2}, $\alpha\in L^2_2(X)$ implies $\alpha, e^{2\alpha}-1\in L^p(X)$ for $2\leq p<\infty$. In particular, $e^{2\alpha}, \gamma\in L^4(X)$. Since $\sigma_1\in L^\infty(X)$ and $L^4\times L^4\embed L^2$, we have 
\[
e^{2\alpha}|\sigma_1|^2\gamma=|\sigma_1|^2\gamma+(e^{2\alpha}-1)|\sigma_1|^2\gamma\in L^2(X). 
\]

Note that $\mathcal{D}_\alpha\mu$ is clearly symmetric. It suffices to show that it agrees with its adjoint. We need to show $\gamma\in L^2(X)$ and $\mathcal{D}_\alpha\mu(\gamma)\in L^2(X)$ implies $\gamma\in L^2_2(X)$. This can be done directly, but we proceed using Friedrichs extension theorem. Define the norm $\A$ on $\SC_c^\infty(X)$ by the formula
\[
\|x\|_\A^2=\|x\|_2^2+\int_X |dx|^2+|xe^\alpha \sigma_1|^2.
\]
The Hilbert space $H$ obtained by the completion with respect to $\|\cdot\|_\A$ is embedded as a subspace of $L^2(X)$. Then by Friedrichs extension theorem, $\mathcal{D}_\alpha\mu$ is self-adjoint on the space
\[
D=\{x\in H: \exists\ C>0,  \langle x, y\rangle_\A\leq C\|y\|_2, \text{ for any $y\in H$}\}. 
\]

We need to identify $D$ with $L_2^2(X)$. It is clear that $L_2^2(X)\subset D$. For the reversed inclusion, take any $y\in \SC_c^\infty(X)$ and $x\in D$. Then integration by parts shows
\[
\langle x, y\rangle_\A=\langle \mathcal{D}_\alpha\mu(x),y\rangle_2,
\] 
and by Riesz representation theorem, $\mathcal{D}_\alpha\mu(x)\in L^2(X)$. Since $x\in H\embed L^2_1(X)$ and $L^2_1(X)\embed L^4(X)$, $e^{2\alpha}|\sigma_1|^2x\in L^2(X)$. Therefore, $\Delta_X x=\mathcal{D}_\alpha\mu(x)-e^{2\alpha}|\sigma_1|^2x\in L^2. $ This completes the proof of the lemma. 
\end{proof}

By lemma \ref{SA}, $y\colonequals \mu(\alpha_\infty)$ lies in the domain of  the adjoint operator $(D_\mu)^*=D_\mu$, so $y\in L^2_2(X)$. Let $\gamma=y$ in (\ref{critical}) and integration by parts shows:
\[
0=\|dy\|_2^2+\|ye^\alpha\sigma_1\|_2^2.
\]

Therefore, $\mu(\alpha_\infty)=y\equiv 0$.
\end{proof}

As long as the a priori estimate, Theorem \ref{apriori3}, is established, the proof of Theorem \ref{existence} is quite formal. We will tackle this technical theorem in the next section. 

\subsection{Smoothness and uniqueness}\label{52}
For the rest of the section, we prove the smoothness and uniqueness of the solution obtained in Theorem \ref{existence}.

\begin{lemma}\label{smoothness}
The solution $\talpha=\alpha+\alpha_0$ obtained in the proof of Theorem \ref{existence} is smooth. 
\end{lemma}
\begin{proof}
Note that the equation $\mu(\alpha)=0$ is equivalent to 
\begin{equation}\label{elliptic}
\Delta_X\alpha=-\mu(0)-\half (e^{2\alpha}-1)|\sigma_1|^2.
\end{equation}

Both $\mu(0)$ and $|\sigma_1|^2$ are smooth functions on $X$. Since $\alpha\in L^2_2(X)$, Trudinger's inequality (Theorem \ref{A2}) implies $e^{2\alpha}-1\in L^p(X)$ for any $2\leq p<\infty$. Then elliptic regularity shows 
\[
\alpha\in L^p_{2,loc}(X)
\] 
for any $2\leq p<\infty$, so $\alpha, e^{2\alpha}\in L^\infty_{loc}(X)$. This implies
\[
d(e^{2\alpha}-1)=2e^{2\alpha}d\alpha\in L^2_{loc}(X),
\]
and 
\[
\nabla^2(e^{2\alpha}-1)=4e^{2\alpha}d\alpha\otimes d\alpha+2e^{2\alpha}\nabla^2\alpha\in L^2_{loc}(X).
\]

We use induction to prove that for each $k\geq 2$, 
\[
\alpha,e^{2\alpha}\in L^2_{k,loc}(X).
\]
The initial step $k=2$ is done. Suppose the statement is true for some $k\geq 2$. By elliptic regularity and the induction hypothesis that $e^{2\alpha}\in L^2_{k,loc}$, (\ref{elliptic}) implies $\alpha\in L^2_{k+2,loc}$. To verify $e^{2\alpha}\in L^2_{k+2,loc}$, it suffices to note that $L^2_{k+2}$ is an algebra for any $k>0$ and the expansion 
\[
e^{2\alpha}=\sum_{m=0}^\infty \frac{(2\alpha)^m}{m!}
\]
converges in $L^2_{k+2,loc}$-topology for any $\alpha\in L^2_{k+2,loc}$. The induction step is accomplished. Since $\alpha\in L^2_{k,loc}$ for any $k\geq 2$, $\alpha$ is smooth.
 \end{proof}

\begin{lemma}
The solution $\talpha=\alpha+\alpha_0$ obtained in Theorem \ref{existence} is also unique.
\end{lemma}
\begin{proof} 
	
Let $\alpha'=\alpha+\alpha_0$ be the solution obtained in the proof of Theorem \ref{existence}. We need to show for any $\alpha''\in \SC^\infty(X)$ such that $(A,\sigma)=e^{\alpha''}\cdot(A_0,\sigma_0)$ solves equation (\ref{firsteq}) and has finite analytic energy, $\alpha''=\alpha'$. 

Let $\gamma=\alpha''-\alpha'$, then $\gamma$ is smooth and  $\mu(\gamma+\alpha)=\mu(\alpha)=0$. This shows
\begin{equation}\label{gamma}
\Delta_X\gamma+\half e^{2\alpha}(e^{2\gamma}-1)|\sigma_1|^2=0.
\end{equation}

Since $(A,\sigma)=e^{\alpha''}\cdot (A_0,\sigma_0)=e^{\alpha''-\alpha_0}\cdot (A_1,\sigma_1)$ has finite analytic energy, by the proof of Lemma \ref{boundedness}, $\|\alpha''-\alpha_0\|_{L^\infty(\{z\}\times \Sigma)}\to 0$ as $z\to\infty$. The reason is that when we properly translate $(A,\sigma)$ to the origin and get a sequence of solutions $(A_n,\sigma_n)$ on $X'=B(0,10)\times\Sigma$, we impose Coulomb-Neumann gauge fixing condition (with respect to $(A_0,\gamma_d)$) on $B(0,10)\subset \C$. As $n\to\infty$, these solutions will converge in the interior in $\SC^\infty$-topology to $(A_0,\gamma_d)$. This gives the desired convergence.

Therefore, we have, in addition to (\ref{gamma}), that
\begin{equation}\label{boundarycondition}
\gamma=(\alpha''-\alpha_0)-\alpha\to 0
\end{equation}
as $|z|\to\infty$. Then the maximal principle implies $\gamma\equiv 0$. Indeed, suppose $\gamma>0$ somewhere, then by (\ref{boundarycondition}), it attains its maximum at some point $p\in X$ and hence $\Delta_X \gamma (p)\geq 0$. If $\Delta_X\gamma(p)>0$, then $(\ref{gamma})$ is violated at $p$. The case when $\Delta_X\gamma(p)=0$ is tricker: we need to add a perturbation. For details, see \cite[Chapter VI.3, Proposition 3.3]{JT}. This shows $\gamma\leq 0$. Similarly, by analyzing the minimum of $\gamma$, we conclude $\gamma\geq 0$, so $\gamma\equiv 0$. 
\end{proof}
\section{Proof of Theorem \ref{apriori3}}\label{sec6}
Theorem \ref{apriori3} states that the $L^2$-norm of $\mu(\alpha)$ controls the $L^2_2$-norm of $\alpha$ for any $\alpha\in L_2^2(X)$. The proof is achieved by two steps:

\textit{Step 1}. Estimate $\|\Delta_X\alpha\|_2$.

\textit{Step 2}. Estimate $\|\alpha\|_2$. 

The first step is a consequence of the energy formula (\ref{analyticenergy}), but we need to work very carefully. The second step is trickier: it involves decomposing the function $\alpha$ into its high frequency and low frequency modes.

If one carries out the same proof for the vortex equation on $\Sigma$ and $\C$, an a priori estimate like Theorem \ref{apriori3} will be needed as well. For details, see Theorem \ref{apriori1} and \ref{apriori2}.  In both cases, the first step is trivial. On $\Sigma$, the second step is false, but we can still make it work by using a smaller variational space. On $\C$, the second step is processed similarly as we did here, yet it is simpler. 

We establish Theorem \ref{apriori3} in a sequence of lemmas. We start with \textit{Step 1.} 
\begin{lemma}\label{Laplacianestimate}
There exist $a,b>0$ such that for any $\alpha\in L^2_2(X)$ with $\E(\alpha)<C$, we have
\[
\|\Delta_X\alpha\|_{L^2(X)}^2<aC+b. 
\]
\end{lemma}
\begin{proof}
It suffices to find $a,b>0$ so that 
\[
\|\Delta_\C\alpha\|_{L^2(X)}^2, \|\Delta_\Sigma\alpha\|_{L^2(X)}^2<aC+b. 
\]

By (\ref{analyticenergy}), we know that 
\[
\E(\alpha)>\int_X |F_{A_\alpha}^\C|^2=\int_X |\Delta_\C\alpha+\Delta_\C\alpha_0|^2.
\]

Now, the elementary inequality,
\begin{equation}\label{peterpaul}
(a+b)^2\geq \half a^2-b^2
\end{equation}
implies that $\|\Delta_\C\alpha\|_2^2\leq 2(\E(\alpha)+\|\Delta_\C\alpha_0\|^2_2).$

To analyze the term $\|\Delta_\Sigma\alpha\|_2$ is harder. The energy formula (\ref{analyticenergy}) also implies
\begin{align*}
\E(\alpha)&>\int_X|iF_{A_\alpha}^{\Sigma}+\half K+\half |\sigma_\alpha|^2|^2\\
&= \int_X|(\Delta_\Sigma\alpha+\half(e^{2\alpha}-1)|\sigma_1|^2)+\half (|\sigma_1|^2-|\gamma_d|^2)|^2.
\end{align*}

Then the inequality (\ref{peterpaul}) implies
\[
2\E(\alpha)+2\||\sigma_1|^2-|\gamma_d|^2\|_2^2\geq  \int_X|\Delta_\Sigma\alpha+\half(e^{2\alpha}-1)|\sigma_1|^2|^2.
\]

Therefore, Lemma \ref{Laplacianestimate} follows from a fiber-wise estimate and it is the content of the next lemma.\end{proof}

\begin{lemma}\label{fiber} For each $z\in \C$, let $(B_1(z), \sigma_1(z))=(\nabla_{A_1}^\Sigma, \sigma_1(z))\in \SC(\Sigma, L)$ be the configuration on the fiber $\{z\}\times \Sigma$. There exists $c>0$ such that for any $\alpha\in L^2_2(\Sigma)$ and $z\in\C$, 
\[
\int_\Sigma|\Delta_\Sigma\alpha+\half(e^{2\alpha}-1)|\sigma_1(z)|^2|^2\geq c\int_\Sigma |\Delta_\Sigma\alpha|^2.
\]
\end{lemma}

\Remark  Though the proof below is messy, the underlying idea should be clear. This lemma is true because for each individual $\sigma_1(z)$, the statement is true for its linearized operator. Equations (\ref{8}) and (\ref{18}) below are steps where we pass to linearized operators. This resolves the case when $\|\Delta_\Sigma\alpha\|_2$ is small. When $\|\Delta_\Sigma\alpha\|_2$ is large, this estimate is true due to the energy equation. 
\begin{proof} The family of configurations $(B_1(z), \sigma_1(z))$ admits a natural compactification; we extend it to $\CP$ by setting
	\[
	(B_1(\infty), \sigma_1(\infty))=(B_0,\gamma_d). 
	\]
	
In particular, there is $M>0$ such that 
\[
\|\sigma_1(z)\|_{L^2_2(\Sigma)}^2+\|F_{B_1}(z)\|_{L^2_2(\Sigma)}^2<M
\]
holds for any $z\in \C$. What will be frequently used below is the Sobolev embedding theorem:
\[
L_2^2(\Sigma)\embed L^\infty(\Sigma),
\]
and the fact that this embedding is compact.
\bigskip

To start, we have a nice energy equation associated to the vortex equation (\ref{vortex}),  as a special case of Lemma \ref{positivity} or Formula (\ref{analyticenergy}):
	\begin{multline*}
	\int_\Sigma 2|\bpartial_B\sigma|^2+ |*iF_B+\half K+\half|\sigma|^2|^2\\
	=\int_\Sigma |\nabla_B\sigma|^2+|*iF_B+\half K|^2+\frac{1}{4}(|\sigma|^2+K)^2-\frac{1}{4}K^2. 
	\end{multline*}

We apply this equation to $(B,\sigma)=e^{\alpha}\cdot (B_1,\sigma_1)=(B_1+i*d_\Sigma\alpha,e^\alpha\cdot \sigma_1)$ and get 
	\begin{multline}\label{1}
\int_\Sigma  |\Delta_\Sigma\alpha+ *iF_{B_1}+\half K+\half|\sigma_1|^2e^{2\alpha}|^2\\
\geq \int_\Sigma |\Delta_{\Sigma}\alpha+*iF_{B_1}+\half K|^2-\frac{1}{4}K^2,
\end{multline}
using the fact that $\bpartial_{B_1}\sigma_1=0$, so $\bpartial_B\sigma=0$. Then the Cauchy-Schwartz inequality and (\ref{peterpaul}) imply
\begin{multline*}
2\int_\Sigma  |\Delta_\Sigma\alpha+\half(e^{2\alpha}-1)|\sigma_1|^2|^2+2\int_\Sigma  |*iF_{B_1}+\half K+\half|\sigma_1||^2\geq \text{LHS of (\ref{1})},\\
\end{multline*}
and 
\[
\text{RHS of (\ref{1})}\geq \half \int_\Sigma |\Delta_{\Sigma}\alpha|^2-\int_\Sigma (|*iF_{B_1}+\half K|^2+\frac{1}{4}K^2). 
\]

Finally, we get 
\begin{equation}\label{5}
4\int_\Sigma|\Delta_\Sigma\alpha+\half(e^{2\alpha}-1)|\sigma_1|^2|^2+N>\int_\Sigma |\Delta_\Sigma\alpha|^2,
\end{equation}
for some $N>0$ independent of $z\in \CP$. 

\smallskip

Suppose Lemma \ref{fiber} is violated. Then, for each $n>0$, there is $(\alpha_n,z_n)\in L^2_2(\Sigma)\times \C$ such that 
\begin{align}\label{6}
\int_\Sigma|\Delta_\Sigma\alpha_n+\half(e^{2\alpha_n}-1)|\sigma_1(z_n)|^2|^2&<\frac{1}{n}\int_\Sigma |\Delta_\Sigma\alpha_n|^2. 
\end{align} 

Let $\sigma_n=\sigma_1(z_n)$. By (\ref{5}), we must have $\int_\Sigma |\Delta_\Sigma\alpha_n|^2<\frac{n}{n-4} N\leq 5N$ when $ n\geq 5$. This shows the sequence $\beta_n\colonequals \Delta_\Sigma\alpha_n$ is bounded in $L^2(\Sigma)$ and we can find a weakly convergent subsequence. Since this is a compact family of configurations,  we can further assume that for this subsequence $z_n\to z_\infty$, so $\{|\sigma_n|^2\}$ is convergent in $L^\infty$. Write $\alpha_n=G\beta_n+\delta_n$, where $G: L^2(\Sigma)^\perp\to L^2_2(\Sigma)^\perp$ is the Green operator and $\delta_n$ is the average of $\alpha_n$ on $\Sigma$. By (\ref{6}),  we know that the sequence of functions
\begin{equation}\label{36}
g_n\colonequals\Delta_\Sigma\alpha_n+\half(e^{2\alpha_n}-1)|\sigma_n|^2=\beta_n-\half |\sigma_n|^2+\half |\sigma_n|^2e^{2G\beta_n}e^{2\delta_n}
\end{equation} 
converges strongly to $0$ in $ L^2(\Sigma)$. Since $L_2^2\embed L^\infty(\Sigma)$ is compact, $\{G\beta_n\}$ and hence $\{e^{G\beta_n}\}$ is convergent in $L^\infty(\Sigma)$. So far, we have shown
\begin{align*}
\beta_n&\xrightarrow{w-L^2} \beta_\infty, &|\sigma_n|^2&\xrightarrow{s-L^\infty} |\sigma_\infty|^2\colonequals|\sigma(z_\infty)|^2, \\
e^{2G\beta_n}&\xrightarrow{s-L^\infty} e^{2G\beta_\infty}, & g_n&\xrightarrow{s-L^2} 0,
\end{align*}
for some $\beta_\infty\in L^2(\Sigma)$. The relation (\ref{36}) then implies
\begin{equation}\label{weakconvergence}
 \{|\sigma_n|^2e^{2\delta_n}\} \text{ is weakly convergent in } L^2(\Sigma). 
\end{equation}

 This convergence can be made to be strong; indeed, write $\sigma_n=r_n\sigma_n'$ with $r_n=\|\sigma_n\|_2$ and $\|\sigma_n'\|_2=1$. Then $r_n> 0$ for any $n$. If $r_n=0$ for some $n$, then $\sigma_1(z_n)=\sigma_n=0$ in $(\ref{6})$, and 
  \[
  \int_\Sigma |\Delta_\Sigma\alpha_n|^2< \frac{1}{n}\int_\Sigma |\Delta_\Sigma\alpha_n|^2,
  \]
  which is absurd. By passing to a subsequence, we can assume $\{\sigma'_n\}$ converges strongly, since
  \[
  \{\frac{\sigma_1(z)}{\|\sigma_1(z)\|_2}: z\in\CP, \|\sigma_1(z)\|_{L^2(\Sigma)}\neq 0\}
  \]
  forms a compact family. Now $\{a_n\colonequals r_ne^{\delta_n}=\|e^{\delta_n}\sigma_n\|_2\}$ is a sequence of bounded real numbers by (\ref{weakconvergence}), so there is a converging subsequence. We may assume $\{r_n\}$ converges as well. In all, 
  \begin{align*}
  a_n&\to a_\infty, &r_n&\to r_\infty, \\
\sigma_n'&\xrightarrow{s-L^2} \sigma_\infty', &|\sigma_n|^2e^{2\delta_n}=a_n^2|\sigma_n'|^2&\xrightarrow{s-L^2} a_\infty^2 |\sigma_\infty'|^2,
  \end{align*}
for some $a_\infty, r_\infty\geq 0$ and $\sigma_\infty'\in L^2(\Sigma)$.

All these things imply that $\beta_n\xrightarrow{s-L^2} \beta_\infty$ as $n\to \infty$ by (\ref{36}). There are two cases to be dealt with:

\bigskip
\textit{Case 1}. If $r_\infty=0$, i.e. $\sigma_\infty\colonequals \sigma(z_\infty)=0$. Let $n\to\infty$ in (\ref{36}), so 
\[
\beta_\infty=-\half a_\infty^2|\sigma_\infty'|^2e^{Gb_\infty}\leq 0.
\]
But $\int_\Sigma \beta_\infty=0$ and hence $a_\infty=0,\ \beta_\infty\equiv 0$. Write $\beta_n=s_n\beta_n'$ with $\|\beta_n'\|_2=1$. Divide (\ref{36}) by $s_n$, 
\begin{equation}\label{8}
\frac{g_n}{s_n}= \beta_n'+\half \frac{e^{2s_nG\beta_n'}-1}{s_n}a_n^2|\sigma_n'|^2+\half \frac{a_n^2-r_n^2}{s_n}|\sigma_n'|^2.
\end{equation}

By (\ref{6}), $\|g_n/s_n\|_2^2<1/n\cdot \|\beta_n'\|_2^2\to 0$. The second term on the right hand side of (\ref{8}) converges to $0$ in $L^1(\Sigma)$ as $a_\infty=\lim a_n=0$. Since $\int_\Sigma\beta_n'=0$ for  each $n$, integrating $(\ref{8})$ over $\Sigma$ yields
\[
x_n\colonequals \frac{a_n^2-r_n^2}{s_n}=2\int_\Sigma \frac{g_n}{s_n} -\beta_n'-\half \frac{e^{2s_nG\beta_n'}-1}{s_n}a_n^2|\sigma_n'|^2\to 0
\] 
as $n\to \infty$. Finally, let $n\to \infty$ in (\ref{8}), then
\[
\beta_n'\xrightarrow{s-L^2} 0.
\]
This is impossible, as $\lim \|\beta_n'\|_2=1$.

\textit{Case 2}. If $r_\infty\neq 0$, then $\{\delta_n\}$ has a finite limit, say, $\lim \delta_n=\delta_\infty\in\R$. Therefore,  $\lim\alpha_n=\alpha_\infty\colonequals G\beta_\infty+\delta_\infty$ and 
\begin{align*}
0&=\langle\Delta_\Sigma\alpha_\infty+\half(e^{2\alpha_\infty}-1)|\sigma_\infty|^2,\alpha_\infty\rangle\\
&=|\nabla \alpha_\infty|^2+\half r_\infty^2\int_\Sigma (e^{2\alpha_\infty}-1)\alpha_\infty|\sigma_\infty'|^2.
\end{align*}

Since $r_\infty>0$, we have $\alpha_\infty\equiv 0.$ Now, write $\alpha_n=t_n\alpha_n'$ with $\|\alpha_n'\|_{L^2_2}\equiv 1$. Choose a subsequence so that $\alpha_n'\xrightarrow{w-L^2_2}\alpha_\infty'$. Divide (\ref{36}) by $t_n$:
\begin{equation}\label{18}
\frac{g_n}{t_n}= \Delta_\Sigma\alpha_n'+\half \frac{e^{2t_n\alpha_n'}-1}{t_n}|\sigma_n|^2.
\end{equation}
By (\ref{6}), $\|g_n/t_n\|_2^2<1/n\cdot \|\Delta_\Sigma\alpha_n'\|_2^2\to 0$. The last term in (\ref{18}) converges in $L^\infty$ to $\alpha_\infty'|\sigma_\infty|^2$ since $t_n\to 0$. This shows $\alpha_n'\xrightarrow{s-L^2_2}\alpha_\infty'$. Let $n\to \infty$ in (\ref{18}):
\[
0=\Delta_\Sigma\alpha_\infty'+|\sigma_\infty|^2\alpha_\infty'.
\]

Since $|\sigma_\infty|^2\not\equiv 0$, the operator $\Delta_\Sigma+|\sigma_\infty|^2$ is injective on $L^2_2$, so $\alpha_\infty'\equiv 0$. But  $\|\alpha_\infty'\|_{L^2_2}=\lim \|\alpha_n'\|_{L^2_2}=1$, which is absurd. 

This completes the proof of Lemma \ref{fiber}. 
\end{proof}

\begin{proof}[Proof of Theorem \ref{apriori3}]
In light of Lemma \ref{Laplacianestimate}, it suffices to work out \textit{Step 2}: find a  function $\eta:\R^+\to\R^+$ such that 
\[
\|\alpha\|_2<\eta(C)
\]
for any $\alpha\in L^2_2(\Sigma)$ with $\E(\alpha)<C$. We know that for some $a,b>0$,
\begin{equation}\label{9}
\|\Delta_X \alpha\|_2^2, \|(e^{2\alpha}-1)|\sigma_1|^2\|_2^2<aC+b.
\end{equation}

Write $\alpha=\alpha_1+\alpha_2$ with $\alpha_1$ constant on each fiber $\{z\}\times \Sigma$. In other word, 
\[
\alpha_1(z)=\int_{\{z\}\times\Sigma}\alpha(z,\cdot),
\]
and $\alpha_2(z)$ is orthogonal to constant functions on each fiber. The high frequency part $\alpha_2$ is relatively easy to control:
\[
\lambda_1\|\alpha_2\|_2\leq \|\Delta_\Sigma\alpha_2\|_2=\|\Delta_\Sigma\alpha\|_2\leq \sqrt{aC+b},
\]
where $\lambda_1$ is the first positive eigenvalue of $\Delta_\Sigma$.  
\bigskip

To work out $\|\alpha_1\|_2$ is harder. We will decompose $\C$ as the union of the good set $A_1$ and the  bad set $A_2$. A point $z\in \C$ lies in the good set $A_1$ if over the fiber $\{z\}\times\Sigma$, either of the situations occurs:
\begin{enumerate}
\item\label{case1} $\|\alpha_1\|_2\ll \|\alpha_2\|_2$. This means $\alpha$ has large fluctuation on that fiber, so $\|\alpha\|_2\lesssim \|\Delta_{\Sigma}\alpha\|_2$.

\item\label{case2} $\|\alpha_1\|_2\gg \|\alpha_2\|_2$, but $\alpha_1(z)>-1$. This means $\alpha$ is almost a constant function on that fiber and its value is not very negative, so $|\alpha|\lesssim |e^{\alpha}-1|$.
\end{enumerate}

The actual definition of $A_1$ below is based on this intuition, but stated in a slightly different way. Let $A_2$ be the complement of $A_1$. On the good set $A_1$, we control $\|\alpha_1\|_{L^2(A_1\times\Sigma)}$ by terms in (\ref{9}). For the bad set $A_2$, we will control its area and show that $|\alpha_1|^2$ cannot concentrate on $A_2$.

\begin{lemma}\label{lemma}
There is a constant $L>0$ such that for any $\alpha\in L^2_2(\Sigma)$ such that
\[
\int_\Sigma |\alpha|^2\geq L\int_\Sigma |\Delta_\Sigma\alpha|^2,
\]
we have $\|\alpha-\alpha_1\|_\infty<\half|\alpha_1|$ where $\alpha_1\in \R$ is the average of $\alpha$ on $\Sigma$.
\end{lemma}
\begin{proof}
Write $\alpha=\alpha_1+\alpha_2$. Then we have
\begin{align*}
\|\alpha_2\|_{L^2_2}^2\leq (1+\frac{1}{\lambda_1^2})\int_\Sigma |\Delta_\Sigma\alpha|^2&\leq \frac{1}{L} (1+\frac{1}{\lambda_1^2})\|\alpha\|_2^2\\
&\leq \frac{1}{L} (1+\frac{1}{\lambda_1^2})(Vol(\Sigma)|\alpha_1|^2+\|\alpha_2\|_{L^2_2}^2).
\end{align*}

If $L>2(1+\frac{1}{\lambda_1^2})$, then we do rearrangement and obtain
\[
C|\alpha_1|>\sqrt{L}\|\alpha_2\|_{L_2^2}>\frac{\sqrt{L}}{C_1}\|\alpha_2\|_\infty.
\]

Here, $C_1$ is the constant in the Sobolev embedding $L_2^2(\Sigma)\embed L^\infty(\Sigma)$. Now it suffices to take $L>\max\{(2CC_1)^2,2(1+\frac{1}{\lambda_1^2})\}.$ 
\end{proof}

Now we decompose $\C$ into good and bad sets:
\begin{align*}
	A_1&=\{z\in\C|\int_{\{z\}\times\Sigma} |\alpha|^2< L\int_{\{z\}\times\Sigma} |\Delta_\Sigma\alpha|^2+M \int_{\{z\}\times\Sigma}|e^{2\alpha}-1|^2|\sigma_1|^4\}\\
		A_2&=\C\setminus A_1.
\end{align*}
In this definition, $L$ is the constant in Lemma \ref{lemma} and $M$ is a large constant to be determined later. Note that if $z\in A_1$, then either 
\[
\int_{\{z\}\times\Sigma}  |\alpha|^2< 2L\int_{\{z\}\times\Sigma} |\Delta_\Sigma\alpha|^2
\]
or 
\[
\int_{\{z\}\times\Sigma}  |\alpha|^2< 2M \int_{\{z\}\times\Sigma}|e^{2\alpha}-1|^2|\sigma_1|^4,
\]
which correspond to Case (\ref{case1}) and Case (\ref{case2}) respectively. 

Our goal is to show for some $\eta_1(C)>0$,
\begin{equation}\label{19}
\int_\C |\Delta_\C\alpha_1|^2, \int_{A_1}|\alpha_1|^2,\  \Area(A_2)<\eta_1(C). 
\end{equation}

The first follows from (\ref{9}). The good set $A_1$ is easy to handle: $\int_{A_1}|\alpha|^2<\eta_2(C)$ for some $\eta_2$, again, by (\ref{9}). As for the bad set, we need to analyze the zero locus of $\sigma_1$. Set 
\[
Z_\epsilon(\sigma_1)=\{p=(z,x)\in X: |\sigma_1(p)|^2<\epsilon\},
\]
and 
\[
Z_\epsilon=\{x\in\Sigma: |\gamma_d(x)|^2<2\epsilon\}.
\]

We choose $\epsilon$ to be a small number so that $\Area(Z_\epsilon)<\half \Area(\Sigma)$.
Since $\sigma_1(z)$ as sections on $\Sigma$ approach $\gamma_d$ in $L^\infty$ norm as $z\to\infty$, for some large number $R(\epsilon)>0$, we know
\[
Z_\epsilon(\sigma_1)\subset B(0,R)\times \Sigma\cup \C\times Z_\epsilon.
\]

Now take any $z\in A_2-B(0,R)$. By Lemma \ref{lemma}, we have
\[
\half |\alpha_1(z)|<|\alpha(z,x)|<\frac{3}{2} |\alpha_1(z)|
\]
for any $x\in \Sigma$. If $\alpha_1(z)>-1$, then $\alpha(z,x)>-2$. This means over the fiber $\{z\}\times \Sigma$, we have 
\begin{align*}
\int_\Sigma |\alpha|^2\leq \frac{9}{4}\int_\Sigma |\alpha_1|^2 &\leq \frac{9}{2}\int_{Z_\epsilon^c} |\alpha_1|^2\leq 18\int_{Z_\epsilon^c} |\alpha|^2\\
&\leq \frac{18\times 25}{\epsilon^2}\int_{Z_\epsilon^c} |e^{2\alpha}-1|^2|\sigma_1|^4\leq \frac{450}{\epsilon^2}\int_\Sigma |e^{2\alpha}-1|^2|\sigma_1|^4,
\end{align*}
where the penultimate inequality comes from
\[
|x|<5|e^x-1|
\]
when $x>-4$. Take $M=450/\epsilon^2$. We conclude that if $z\in A_2-B(0,R)$, then $\alpha_1(z)\leq -1$ and hence $\alpha(z,x)\leq -\half$. This shows
\[
 \int_{\{z\}\times\Sigma} |e^{2\alpha}-1|^2|\sigma_1|^4\geq \half Vol(\Sigma)\cdot |1-e^{-1}|^2\epsilon^2.
\]

As a consequence, we obtain,
\[
\Area(A_2)\leq \Area(B(0,R))+\frac{\||e^{2\alpha}-1||\sigma_1|^2\|_2^2}{C_1}\leq \eta_3(C). 
\]

(\ref{19}) is proven. The next step is to control $\|\alpha_1\|_{L^2(\C)}$. This is closely related to uncertainty principle: if $|\alpha_1|^2$ concentrates on a region of finite area, say $A_2$, then its Fourier transformation cannot concentrate near the origin. Thus, $L^2$ norm of $\alpha_1$ is controlled by the $L^2$ norm of $\Delta \alpha_1$.

Theorem \ref{apriori3} then follows from the next lemma by setting $n=2$, $f=\alpha_1$, $S=\eta_3(C)$ and $E=A_2$. 
\end{proof}

\begin{lemma}\label{ABlemma}
Suppose subset $E\subset \R^n$ is measurable and its volume $m(E)\leq S$. Then for some $C(S)>0$, we have for any $f\in L^2_2(\R^n)$,
\[
\|f\|_2\leq C(\|f\|_{L^2(E^c)}+\|\Delta f\|_2).
\]
\end{lemma}  

\begin{proof} Let $\chi$ is a cut-off function with $\chi\equiv 1$ on $B(0,r)$ and $\supp\chi\subset B(0,2r)$. Here, $r$ is a small number to be determined later.
We decompose $f$ as the sum of low-frequency and high-frequency parts: $f=f_L+f_H$ where
\[
\hat{f}_L=\chi\hat{f}, \hat{f}_H=(1-\chi)\hat{f}. 
\]
 
The high frequency part is easy to deal with:
\[
\|f_H\|_2=\|\hat{f}_H\|_2\le\frac{1}{r^2}\||\xi|^2\hat{f}_H\|_2\leq \frac{1}{r^2}\|\Delta f\|_2. 
\]

Write $T=\|f_L\|_2$. We control $L^\infty$-norm of $\nabla f_L$ in terms of $T$:
\[
\|\nabla f_L\|_\infty\leq  \||\xi|\hat{f}_L\|_1\leq \||\xi|\cdot \chi_{B(0,2r)}\|_2\|\hat{f}_L\|_2\leq rC_1T.
\]

To bound $T^2$, one intermediate step is to show that 
\begin{equation}\label{6.4-1}
T^2\leq C_2\int_{E^c}|f_L|^2
\end{equation}
for some $C_2>0$ as $|f_L|^2$ cannot concentrate within $E$. To make this idea precise, choose $R>0$ so that $\Vol(B(0,R))>2S=2m(E)$.  If for some $z\in E$, $|f_L(z)|=N>2rC_1TR$, then for any $z'\in B(z,R)$,  $|f_L(z')|>\half N$. This implies
\[
\|f_L\|_{L^2(E^c)}^2\geq\int_{B(z,R)-E}|f_L(z')|^2\geq \frac{\Vol(B(z,R))}{2}\cdot (\frac{N}{2})^2\geq S\cdot (rC_1TR)^2 . 
\]

Therefore, either $T<\|f_L\|_{L^2(E^c)}/(\sqrt{S}rC_1R)$ or 
\[
\|f_L(z)\|_{L^\infty(E)}\leq 2rC_1TR. 
\]
But the second case implies 
\begin{align*}
T^2=\int |f_L|^2&\leq \int_{E^c} |f_L|^2+ \int_{E} |f_L|^2\leq \int_{E^c} |f_L|^2+m(E)\cdot \|f_L(z)\|_{L^\infty(E)}^2\\
&\leq \int_{E^c} |f_L|^2+S\cdot (2rC_1R)^2\cdot T^2.
\end{align*}

It suffices to choose $0<r\ll 1$ such that $ (2C_1R)^2S\cdot r^2<\half$. Therefore, in either case $(\ref{6.4-1})$ holds.

Finally, one notices that $\|f_L\|_{L^2(E^c)}\leq \|f_H\|_{L^2(E^c)}+\|f\|_{L^2(E^c)}\leq\|f_H\|_2+ \|f\|_{L^2(E^c)}.$
\end{proof}

\Remark Lemma \ref{ABlemma} is closely related to the Amrein-Berthier theorem:
\begin{theorem}[Amrein-Berthier \cite{AB}]
	Suppose subsets $E, F\subset \R^n$ are measurable of finite volume, then for some $C(E,F)>0$ and any $f\in L^2(\R^n)$, 
	\[
	\|f\|_2\leq C(\|f\|_{L^2(E^c)}+\|\hat{f}\|_{L^2(F^c)}),
	\]
where $E^c$ and $F^c$ are complements of $E$ and $F$ respectively. 
\end{theorem}

To see their relation, let $F=B(0,1)$ and note that
\[
\|\hat{f}\|_{L^2(F^c)}\leq \||\xi|^2\hat{f}\|_2=\|\Delta_\C f\|_2. 
\]

However, it is not clear from this theorem that for a fixed subset $F$, the constant $C(E, F)$ could be independent of the shape of $E$ and only depends on its volume $m(E)$. Since our problem is simpler, we decided to give a direct proof to Lemma \ref{ABlemma} as above.

\section{Exponential Decay and Power Law Decay}\label{7}
The purpose of section is to prove Theorem \ref{powerlaw} and Theorem \ref{exponentialdecay} which predict power law decay and exponential decay for finite energy monopoles in different cases. We start with the second theorem to explain ideas

\subsection{Exponential Decay} We reformulate Theorem \ref{exponentialdecay} as follows:
\begin{theorem}\label{exponential2}
Let $(B_0,\gamma_d)$ be a solution to the vortex equation $(\ref{vortex2})$. Let $f_0=\sum_{i=0}^d a_iz^i$ be a monic polynomial function on $\C$ and $f=f_0\gamma_d$. Then the solution $(A,\sigma)=e^\talpha\cdot(A_0,\sigma_0)$ obtained in Theorem \ref{existence} converges exponentially to $(B_0,\gamma_d)$ as $|z|\to\infty$, i.e., for any $k\geq 2$, there exists $s(k,A_0,f)$ and $M(k,A_0,f )>0$ such that for any $z\in\C$,
\begin{equation}\label{71}
d_k((\nabla_{A}^\Sigma,\sigma)(z)-(B_0,\gamma_d))<Me^{-s|z|}.
\end{equation}
\end{theorem} 
Recall that the metric $d_k$ on $\SB(\Sigma, L)=\SC(\Sigma,L)/\SG(\Sigma)$ is defined by the formula:
\[
d_k([a],[b])=\min_{u\in \SG(\Sigma)}\|u\cdot a-b\|_{L^2_k(\Sigma)},
\]
for any $a,b\in \SC(\Sigma, L)$. Here, $[a]$ denotes the gauge equivalent class of $a$.

The proof of Theorem \ref{exponential2} relies on exponential decay result for vortices on $\C$. Recall that the classical vortex equation on $\C$ is given by the formula:
\begin{equation}\label{vortex4}
\left\{\begin{array}{r}
*iF_\omega+\half(|\eta|^2-1)=0,\\
\bpartial_\omega\eta=0,
\end{array}
\right.
\end{equation}
where $\omega$ is a smooth unitary connection to the trivial bundle over $\C$ and $\eta$ is a smooth complex-valued function. This equation is invariant under the gauge action of $\SG(\C)=\map(\C,S^1)$. Then Theorem \ref{JT} ([\cite{JT}, p.59, Theorem 1.4]) states that $|F_\omega|=\half ||\eta|^2-1|$ has exponential decay at infinity if $(\omega, \eta)$ is a solution to (\ref{vortex4}) with finite energy.

The proof of Theorem \ref{exponential2} is modeled on the proof of Theorem \ref{existence} and is accomplished in two steps：
\begin{itemize}
\item Find a good approximation $\alpha_0$ to the actual solution $\talpha$. For this part, we need a more clever choice. We employ the existence result (see \cite[Theorem 1]{Taubes} or Theorem \ref{B1}) to find the conformal factor $\alpha_0$ such that $e^{\alpha_0}\cdot (d, f_0)$ solves the vortex equation $(\ref{vortex4})$.
\item Show that the correction term $\alpha=\talpha-\alpha_0$ has exponential decay at infinity. 
\end{itemize}

\begin{proof}[Proof of Theorem \ref{exponential2}]
Let $\alpha_0\in \Gamma(\C,\R)$ be the conformal factor such that $(\nabla_{\omega},\eta)=e^{\alpha_0}\cdot (d, f_0)$ solves (\ref{vortex4}). We regard $\alpha_0$ as a function on $X$ that is constant on each fiber. Let $(A_1,\sigma_1)=e^{\alpha_0}\cdot (A_0,\sigma_0)$, so $\sigma_1=\eta\gamma_d$ and $\Delta_{\C}\alpha_0=i*_\C F_\omega=\half (1-|\eta|^2)$. Since $(B_0,\gamma_d)$ solves the vortex equations (\ref{vortex2}) on $\Sigma$, it follows that 
\[
-\half|\gamma_d|^2=i*_\Sigma F_{B_0}+\half K.
\]

Consider the moment map defined in Definition \ref{momentmap}, then
\begin{align*}
\mu(0)&=\Delta_{\C}\alpha_0+i*_\Sigma F_{B_0}+\half K+\half |\sigma_1|^2\\
&=\half(1-|\eta|^2)-\half|\gamma_d|^2+\half|\gamma_d|^2|\eta|^2\\
&=\half (1-|\eta|^2)(1-|\gamma_d|^2)\in L^2(X).
\end{align*}

Hence, the correction term $\alpha=\talpha-\alpha_0$ is the unique smooth solution to the equation $\mu(\alpha)=0$ such that $\alpha(z)\to 0$ as $z\to\infty$. Its existence is established in Theorem \ref{existence}. Note that $\mu(\alpha)=0$ is equivalent to the equation
\begin{align*}
(\Delta_{\C}+(\Delta_\Sigma+|\gamma_d|^2))\alpha&=-\mu(0)-\half (e^{2\alpha}-2\alpha-1)|\sigma_1|^2+(|\gamma_d|^2-|\sigma_1|^2)\alpha \\
&=k+h(\alpha),
\end{align*}
where $k=(1-|\eta|^2)(|\gamma_d|^2\alpha-\half (1-|\gamma_d|^2))$ and $h(\alpha)=-\half (e^{2\alpha}-2\alpha-1)|\sigma_1|^2$.

Let $\SH$ be the Hilbert space $L^2_k(\Sigma,\C)$ and $L=\Delta_\Sigma+|\gamma_d|^2: \SH\to\SH$ be the unbounded positive self-adjoint operator. Since $(\ref{71})$ is satisfied for the pair $(A_1,\sigma_1)$, in order to prove Theorem \ref{exponential2}, it is sufficient to verify the conditions of the following lemma:
\begin{lemma}\label{formallemma}
Let $\SH$ be a separable Hilbert space and $L: \SH\to\SH$ be a positive self-adjoint operator (possibly unbounded). Suppose there is a smooth function $u: \C\to \SH$ such that 
\begin{enumerate}[(U1)]
\item $(\Delta_{\C}+L)u=k+h(u)$,
\item $\lim_{z\to\infty}\|u\|_{\SH}=0$,
\item $h: \SH\to\SH$ is a continuous map and for some $q>1$ and $C>0$, $\|h(u(z))\|_{\SH}\leq C\|u(z)\|_{\SH}^q$ for any $z\in \C$,
\item $k:\C\to\SH$ is a continuous map such that for some $s,M>0$, $\|k(z)\|_\SH\leq Me^{-s|z|}$ for any $z\in\C$.
\end{enumerate}
Then the function $u$ has exponential decay at $\infty$, i.e., for some $s',M'>0$, $\|u\|_{\SH}<M'e^{-s'|z|}$ for any $z\in\C$. 
\end{lemma}

Property (U2) is by the proof of Lemma \ref{boundedness}. Property (U4) follows from Theorem \ref{JT}.  When $k>1$, $L^2_k(\Sigma,\C)$ is a Banach algebra. To work out property (U3), take $q=2$. Since $\|\alpha(z)\|_\SH$ is uniformly bounded by some number $M_2>0$, it follows that
\begin{align*}
\|e^{2\alpha}-2\alpha-1\|_\SH&\leq \sum_{m=2}^{\infty}\frac{2^m}{m!}\|\alpha\|_{\SH}^m\leq \|\alpha\|_\SH^2 \sum_{m=2}^{\infty} \frac{2^mM^{m-2}_2}{m!}.
\end{align*}

Therefore, it remains to prove Lemma \ref{formallemma}. 
\end{proof}
\begin{proof}[Proof of Lemma \ref{formallemma}]
In order to make things concrete, we first resolve the special case when $\SH=\C$ and $L=\lambda\cdot id_\C$ is a multiple of the identity map ( $\lambda>0$). Then $u:\C\to\C$ is a smooth function. The fundamental solution to the operator $\Delta_{\C}+\lambda$ is given by
\[
K_\lambda(z)=(\frac{1}{|\xi|^2+\lambda})^\vee=\frac{1}{2\pi}K_0(\sqrt{\lambda}|z|),
\]
where $K_0$ is the modified Bessel function of the second kind. For $r>0$ (see \cite[p.377-378]{Handbook}), 
\begin{align*}
K_0(r)&=\int_0^\infty \frac{\cos(rt)}{\sqrt{1+t^2}}dt=\int_0^\infty e^{-r\cosh t}dt\\
&=-\ln(\frac{r}{2})-\gamma+\it{o}(r) \text{ as } r\to 0,\\
&\sim \sqrt{\frac{\pi}{2r}} e^{-r} (1-\frac{1}{8r}+\frac{9}{128r^2}+\cdots) \text{ as }r\to \infty, 
\end{align*}
where $\gamma\approx 0.577$ is the Euler-Mascheroni constant. In particular,
\begin{itemize}
\item $K_\lambda(z)\in L^1(\C)$. Let $M_\lambda\colonequals \int_{\C}|K_\lambda(z)|dz$. Then $M_\lambda=M_1/\lambda.$
\item $K_\lambda(z)$ decays exponentially as $|z|\to\infty$. For any $0<\epsilon\ll 1$, there exists $C_1(\epsilon)>0$ such that for any $r>0$,
\[
\int_{|z|>r}|K_\lambda(z)|=\frac{1}{\lambda}\int_{|z|>\sqrt{\lambda}r}|K_1(z)|\leq \frac{C_1}{\lambda} e^{-(1-\epsilon)\sqrt{\lambda}r}.
\]
\end{itemize}

Let $N_r=\max_{|z|\geq r}\|u(z)\|_\SH$. Property (U2) implies $\lim_{r\to \infty} N_r=0$ and for any fixed $r$, $N_r$ is achieved at some point $z_0$ with $r_0\colonequals |z_0|\geq r$. Let $p=1/q<1$. Since $u$ solves the equation in (U1), we have 
\begin{align}\label{iteration}
N_r&=\|u(z_0)\|_\SH=\|\int_\C K_\lambda(z)(k+h(u))(z_0-z) dz\|_{\SH}\\
&\leq \|\int_{|z|<(1-p)r_0} K_\lambda(z)(k+h(u))(z_0-z) dz\|_{\SH}\nonumber\\
&\qquad+\|\int_{|z|>(1-p)r_0} K_\lambda(z)(k+h(u))(z_0-z) dz\|_{\SH}\nonumber\\
&\leq M_\lambda \max_{|z|\geq pr_0} \|k(z)+h(u)(z)\|_\SH\nonumber\\
&\qquad+\frac{C_1}{\lambda} e^{-(1-\epsilon)\sqrt{\lambda}(1-p)r_0}\cdot \max_{z\in\C} \|k(z)+h(u)(z)\|_\SH\nonumber\\
&\leq  M_\lambda(Me^{-spr_0}+CN^q_{pr_0})+C_2e^{-(1-\epsilon)\sqrt{\lambda}(1-p)r_0}\nonumber\\
&\leq C_3^{q-1}N_{pr}^q+C_4\cdot e^{-s_1r}.\nonumber
\end{align}
where $s_1=\min\{sp, (1-\epsilon)\sqrt{\lambda}(1-p)\}$ and $C_3,C_4>0$ are independent of $r$. The inequality (\ref{iteration}) implies that $u$ has exponential decay, as we explain now. Choose $r\gg 0$ such that $2C_3N_r<1$. Using the relation $p=q^{-1}$, the inequality (\ref{iteration}) implies that for any $n>0$, 
\[
C_3N_{rq^n}\leq (C_3N_{rq^{n-1}})^q+C_5\cdot e^{-s_1 rq^n}.
\]
Let $R=rq^n$. By induction, it is easy to show
\[
C_3N_{rq^n}\leq 2^{q^{n-1}-1}(C_3N_r)^{q^n}+f_n(C_5) e^{-s_1 rq^n},
\]
where $f_n(C_5)$ is a constant that depends on $C_5$. Indeed, the base case when $n=1$ is by (\ref{iteration}) and assuming it holds for $n\geq 1$, then
\begin{align*}
C_3N_{rq^{n+1}}&\leq (C_3N_{rq^n})^q+C_5\cdot e^{-s_1 rq^{n+1}}\\
&\leq  (2^{q^{n-1}-1}(C_3N_r)^{q^n}+f_n(C_5)\cdot e^{-s_1 rq^n})^q+C_5\cdot e^{-s_1 rq^{n+1}}\\
&\leq 2^{q^n-1}(C_3N_r)^{q^{n+1}}+f_{n+1}(C_5)\cdot e^{-s_1 rq^{n+1}},
\end{align*}
where we used the elementary inequality 
\[
(\frac{a+b}{2})^q<\frac{a^q+b^q}{2}
\]
for $a,b>0$ and $q>1$. Note that $f_n$ is determined by the recursion relation
\[
f_1(C_5)=C_5, f_{n+1}(C_5)=2^{q-1}f_n^q(C_5)+C_5.
\]

This recursion will converge when $0<C_5\ll 1$. The limit is going to be the first intersection of the line $y=x$ and the curve $y=2^{q-1}x^q+C_5$ in the first quadrant. We can make $C_5$ small by making $s_1$ smaller and choosing a lager $r$ to start with. Let $\xi=\ln(2C_3N_r)<0$. Therefore, for some $C_6>0$, 
\[
C_3N_R\leq (2C_3N)^{q^n}+C_6 e^{-s_1R}\leq e^{(\xi /r)R}+C_6e^{-s_1R}. 
\]

In general, suppose $rq^n<R<rq^{n+1}$ for some $n\in \Z_+$. Let $R'=rq^n$. Then, 
\[
C_3N_R\leq C_3N_{R'}\leq  e^{(\xi /r)R'}+C_6e^{-s_1R'}\leq e^{(\xi/rq)R}+C_6 e^{-(s_1/q)R}. 
\]

\Remark In order to make this proof work, it suffices to choose $p$ such that $q^{-1}\leq p<1$. The only reason to take $p=q^{-1}$ above is to have a nice-looking proof. It is hard to estimate the optimal exponent for $u$ through this iteration process. However, as long as it is known that $u$ does have exponential decay, one can run through the convolution process and figure out the optimal exponent. The outcome is roughly:
\[
(1-\epsilon)\max_{q^{-1}\leq p<1} \min\{ sp, \sqrt{\lambda}(1-p)\}.
\]

Finally, to work out the general case, we use functional calculus. If the domain $D(L)$ of $L$ embeds compactly into $\SH$, then $L$ has discrete spectrum. In this case, let $0<\lambda_i\leq \lambda_{i+1}$ be eigenvalues of $L$ and $\phi_i$ be their eigenvectors. The fundamental solution of $\Delta_{\C}+L$ can be described nicely as 
\[
K_L(z)=\sum_{i=1}^{\infty} K_{\lambda_i}(z)\phi_i\otimes \phi_i^*\in \Hom(\SH,\SH).
\]
In general, $K_L(z)=\int_{\R} K_\lambda(z) dE_{\lambda}$ where $E_\lambda$'s are spectrum projections associated to $L$. It is clear that $K_L(z)$ is a smooth family of operators on $\C-\{0\}$ and 
\[
\|K_L(z)\|_{\SH\to\SH}\leq |K_{\lambda_1}(z)|
\]
for any $z\neq 0$. Here, $\lambda_1$ is the first positive eigenvalue of $L$. The rest of the proof proceeds as before.
\end{proof}

\subsection{Power Law Decay}

In general, if the polynomial map $f$ is not a product, the corresponding solution $(A,\sigma)$ to the monopole equation will only have power law decay. This is proven by using a generalized version of Lemma \ref{formallemma}:

\begin{lemma}\label{formallemma2} Under the assumption of Lemma \ref{formallemma}, if property (U4) is replaced by 
	\begin{enumerate}[(U4')]
		\item $k:\C\to\SH$ is a continuous map such that for some $M>0$, $\|k(z)\|_\SH\leq M|z|^{-m}$ for any $z\in\C$. 
	\end{enumerate}
	Then for some $M'>0$, $\|u\|_{\SH}<M'|z|^{-m}$ for any $z\in\C$. 
\end{lemma}
\begin{proof} It suffices to modify slightly the proof of Lemma \ref{formallemma}. The convolution process (\ref{iteration})  will give us for any $n\geq 1$,
\begin{equation}\label{iteration2}
	C_3N_{rq^n}\leq (C_3N_{rq^{n-1}})^q+C_5\cdot (rq^n)^{-m}.
\end{equation}
	By induction, we have 
	\[
	C_3N_{rq^n}\leq 2^{q^{n-1}-1}(C_3N_r)^{q^n}+f_n(C_5) (rq^n)^{-m},
	\]
	where $f_n$ is the same function defined in the proof of Lemma \ref{formallemma}. The initial step is automatic. For the induction step, note that
\begin{align*}
	C_3N_{rq^{n+1}}&\leq (C_3N_{rq^n})^q+C_5\cdot (rq^{n+1})^{-m}\\
	&\leq 2^{q^n-1}(C_3N_r)^{q^{n+1}}+2^{q-1}(f_n(C_5)(rq^n)^{-m})^q+C_5(rq^{n+1})^{-m}\\
	&\leq 2^{q^n-1}(C_3N_r)^{q^{n+1}}+f_{n+1}(C_5) (rq^{n+1})^{-m},
\end{align*}
where we need the inequality that $(rq^n)^{-mq}<(rq^{n+1})^{-m}$. It is satisfied when $r\gg 0$. Indeed, we take $r>1$ such that 
\[
(q-1)\ln r>\ln q \ (\geq \ln q+n(1-q)\ln q). 
\]

To make $C_5$ small, we need to replace $m$ by $(1-\epsilon)m$ in (\ref{iteration2}) and choosing a possibly larger $r$ to start. Eventually, we get for some $C_6>0$, 
\[
N_R\leq C_6 R^{-(1-\epsilon)m}
\]
for any $R>0$. This is not our final result yet. Take $\epsilon\ll 1$ so that $(1-\epsilon)q>1$. Let $R=rq^n$ in (\ref{iteration2}):
\[
C_3N_R\leq (C_3C_6)^q(\frac{R}{q})^{-q(1-\epsilon)m}+C_5R^{-m}\leq C_7 R^{-m},
\]
when $R>1$. The proof of Theorem \ref{formallemma2} is now accomplished. 
\end{proof}

The next theorem is a reformulation of Theorem \ref{powerlaw2}:

\begin{theorem}\label{powerlaw2}
	Let $(B_0,\gamma_d)$ be a solution to the vortex equation $(\ref{vortex2})$. Let
	\[
		f=\gamma_d(z^d+a_{d-1}z^{d-1}+\cdots+a_{d-m+1}z^{d-m+1})+\gamma_{d-m}z^{d-m}+\cdots
	\]
	be a polynomial map where $a_i\in\C, d-m+1\leq i\leq d-1$ are complex numbers. Then the solution $(A,\sigma)=e^\talpha\cdot(A_0,\sigma_0)$ obtained in Theorem \ref{existence} converges to $(B_0,\gamma_d)$ at the rate $|z|^{-m}$ as $|z|\to\infty$, i.e., for any $k\geq 2$, there exists $M(k,A_0,\sigma_0 )>0$ such that for any $z\in\C$,
	\[
	d_k((\nabla_{A}^\Sigma(z),\sigma(z))- (B_0,\gamma_d))<M|z|^{-m}.
	\]
\end{theorem} 

\begin{proof} Let $f_0=z^d+a_{d-1}z^{d-1}+\cdots+a_{d-m+1}z^{d-m+1}$. Let $\alpha_0\in \SC^\infty(\C,\R)$ such that $(\nabla_\omega, \eta)\colonequals e^{\alpha_0}\cdot (d, f_0)$ solves the vortex equation (\ref{vortex4}). Let $(A_1,\sigma_1)=e^{\alpha_0}\cdot (A_0,\sigma_0)$ and $\alpha=\talpha-\alpha_0$. By the same computation as in the proof of Theorem \ref{exponential2}, we have 
\[
(\Delta_{\C}+(\Delta_\Sigma+|\gamma_d|^2))\alpha=h(\alpha)+k,
\]
where $h(\alpha)=-\half (e^{2\alpha}-2\alpha-1)|\sigma_1|^2$ and 
\[
k=-\half (1-|\eta|^2)+(|\gamma_d|^2-|\sigma_1|^2)(\alpha+\half).
\]
Since $\sigma_1=\eta\gamma_d+e^{\alpha_0}\gamma_{d-m}z^{d-m}+\cdots$, it follows that
\begin{align*}
|\sigma_1|^2-|\gamma_d|^2&= (|\eta|^2-1)|\gamma_d|^2+2e^{\alpha_0}\re\langle \eta\gamma_d,\gamma_{d-m}z^{d-m}\rangle+\SO(|z|^{-m-1}).
\end{align*}

Note that $e^{\alpha_0}\sim |z|^{-d}$ as $z\to\infty$. This implies $k(z)$ decays at the rate $|z|^{-m}$ at $\infty$. Now we use lemma \ref{formallemma2} to conclude.  
\end{proof}

\appendix

\section{Some analytic results}\label{A}
In this section, we review some analytic results that were used in Section \ref{4} and Section \ref{sec6}. 

In dimension 4, we have Sobolev embedding $L^2_k(\R^4)\embed L^\infty(\R^4)$ if $k>2$. In the borderline case when $k=2$, we have
\[
L^2_2(\R^4)\embed L^p(\R^4)
\]
for any $2\leq p<\infty$. We will prove a weak version of Trudinger's inequality. For the proof of this paper, we will only need these propositions in the special case when $n=2,4$. 

\begin{proposition}[{\cite[Proposition 4.1]{PDEIII}}]\label{A1}There exists $C_n>0$ such that for any $2\leq p<\infty$ and $u\in L^2_{n/2}(\R^n)$,
\begin{equation*}
\|u\|_{L^p(\R^n)}\leq C p^{1/2}\|u\|_{L^2_{n/2}(\R^n)}. 
\end{equation*}
\end{proposition}

\begin{proposition}\label{A2} For any $u\in L^2_{n/2}(\R^n)$ and any $2\leq p<\infty$,
\[
e^u-1\in L^p(\R^n). 
\]
\begin{proof}
By Taylor expansion, Stirling's formula $\sqrt{2\pi} m^{m+\half}e^{-m}\leq m!$ and Proposition \ref{A1} , we have
\begin{align*}
\|e^u-1\|_p&\leq\sum_{m=1}^\infty\frac{1}{m!} \|u^m\|_p=\sum_{m=1}^\infty\frac{1}{m!} \|u\|_{pm}^m\leq  \sum_{m=1}^\infty\frac{1}{m!} C^m_n(pm)^{m/2}\|u\|_{L^2_{n/2}}^m\\
&\leq \sum_{m=1}^\infty \frac{1}{\sqrt{2\pi m}} (\frac{eC_np^{1/2}\|u\|_{L^2_{n/2}}}{m^{1/2}})^m.
\end{align*}
When $m\gg 1$, $(eC_np^{1/2}\|u\|_{L^2_{n/2}})/m^{1/2}<1$, so this series always converges.
\end{proof}
\end{proposition}
\begin{proposition}\label{A3}
The exponential map:
\[
H: L^2_{n/2}(\R^n)\to L^2(\R^n),\ H(u)=e^u-1
\]
is differentiable and $\mathcal{D}_u H(v)=ve^u$. In particular, $H$ is continuous. 
\end{proposition}
\begin{proof}
Let $v\in L^2_{n/2}(\R^n)$. Since $v, e^u-1\in L^4(\R^n)$, 
\[
\|v e^u\|_2=\|v(e^u-1)+v\|_2\leq \|v\|_2+\|v\|_4\|e^u-1\|_4. 
\]
This shows $\mathcal{D}_uH(v)\colonequals ve^u$ is a bounded linear map from $L^2_{n/2}$ to $L^2$.

It suffices to show 
\[
H(u+tv)-H(v)-tve^u=e^u\cdot (e^{tv}-1-tv)\in \SO(t^2).
\]

Using the same argument as above, it suffices to check $\|e^{tv}-1-tv\|_2, \|e^{tv}-1-tv\|_4\in\SO(t^2)$. This is evident from the proof of Proposition \ref{A2}. 
\end{proof}
\begin{proposition}\label{A4}
The exponential map $H: L^2_{n/2}(\R^n)\to L^2(\R^n)$ is weakly continuous. 
\end{proposition}
\begin{proof}
Since $\SC_c^\infty(\R^n)$ is dense in $L^2(\R^n)$, it suffices to show that for any $v\in \SC_c^\infty(\R^n)$ and any sequence
\[
u_k\xrightarrow{w-L^2_{n/2}}u_\infty,
\] 
we have $\langle H(u_k), v\rangle \to \langle H(u_\infty), v\rangle$ as $n\to\infty$. The Sobolev embedding $L^2_{n/2}\embed L^p$ is compact on $B_R\colonequals B(0,R)$ for any $2\leq p<\infty$ and $R>0$. This shows $u_k\to u_\infty$ in $L^p_{loc}$. In addition, for any $m\geq 1$, 
\begin{equation}\label{11}
u_k^m\xrightarrow{L^2_{loc}} u_\infty^m.
\end{equation}
Indeed, by H\"{o}lder's inequality, 
\begin{align*}
\|u_k^m-u_\infty^m\|_{L^2(B_R)}&=\|(u_k-u_\infty)(\sum_{l=0}^{p-1} u_k^l u_\infty^{m-1-l})\|_{L^2(B_R)}\\
&\leq \|u_k-u_\infty\|_{L^4(B_R)}\sum_{m=0}^{p-1}\|u_k\|^l_{L^{4(m-1)}}\|u_\infty\|^{m-1-l}_{L^{4(m-1)}}.\\
&\leq C\|u_k-u_\infty\|_{L^4(B_R)}\to 0.
\end{align*}
In the last step, we used the fact that $\|u_k\|_{L^2_{n/2}}$ and hence $\|u_k\|_{L^p}$ are uniformly bounded for any fixed $2\leq p<\infty$. 

Finally, by the proof of Proposition \ref{A2}, for any $\epsilon>0$, we can find $M>0$ such that for any $k>0$, 
\[
\|e^{u_k}-1-\sum_{m=1}^M \frac{1}{m!} u_k^m\|_2<\epsilon.
\]

Combining with (\ref{11}), this implies $\langle e^{u_n}-1, v\rangle\to \langle e^{u_\infty}-1\rangle. $
\end{proof}

\section{The vortex equation on $\Sigma$}\label{B}

This appendix is based on Grac\'{i}a-Prada's paper \cite{Garcia}. 

For a complex line bundle $L\to \Sigma$, let us fix a hermitian metric and consider the space of smooth unitary connections and smooth sections:
\[
\SC(\Sigma,L)=\A\times \Gamma(\Sigma,L). 
\]

A configuration $(A,\Phi)$ is called a vortex if it solves the vortex equation:
\begin{equation}\label{vortex3}
\left\{\begin{array}{r}
*iF_A+\half(|\Phi|^2-1)=0,\\
\bpartial_A\Phi=0.
\end{array}
\right.
\end{equation}

Each unitary connection $A$ defines a holomorphic structure on $L$ and the second equation of (\ref{vortex3}) is saying $\Phi$ is holomorphic with respect to $A$. 

Consider $\alpha\in \Gamma(\Sigma, \R)$ and $u\in \map(\Sigma, S^1)$. The formula
\begin{align*}
\SG_\C(\Sigma)\ni g=u\cdot e^{\alpha}: \SC(\Sigma,L)&\to \SC(\Sigma,L)\\
(A,\Phi)&\mapsto (A-u^{-1}du+i*d\alpha, ue^{\alpha}\Phi)
\end{align*}
defines the complex gauge transformation on $\SC(\Sigma, L)$, where $u\in \map(\Sigma, S^1)$ and $\alpha\in \Gamma(\Sigma,\R)$. This transformation is designed so that $\bpartial_{g(A)}g(\Phi)=g(\bpartial_A\Phi)$. 

We obtain the usual gauge group $\SG(\Sigma)$ by setting $\alpha=0$. The vortex equation (\ref{vortex3}) is invariant under the action of $\SG(\Sigma)$. 

\begin{theorem}\label{B1}
Suppose $0<\deg L\colonequals c_1(L)[\Sigma]<\frac{1}{4\pi} Vol(\Sigma)$. Then for any effective divisor $D=\sum n_iz_i$ with $\deg D=\deg L$, there is a unique solution $(A,\Phi)$ to the equation $(\ref{vortex3})$ up to gauge such that $Z(\Phi)=D$. 
\end{theorem}

Given any effective divisor $D=\sum n_iz_i$ with $\deg D=\deg L: =c_1(L)[\Sigma]$, $D$ determines a holomorphic structure $\bpartial_D$ and a canonical holomorphic section $\Phi_0$ with respect to $\bpartial_D$. The pair $(\bpartial_D,\Phi_0)$ is unique up to the action of $\SG_\C(\Sigma)$. We fix a representative $(\bpartial_D,\Phi_0)$. The Chern connection $A_0$ is the unique unitary connection on $L$ such that $\bpartial_{A_0}=\nabla^{0,1}_{A_0}=\bpartial_D$. Our goal is to find another configuration $(A,\Phi)$ obtained from $(A_0,\Phi_0)$ by applying an element in $\SG_\C(\Sigma)$ such that the first of (\ref{vortex3}) is satisfied.

Since we are interested in solutions modulo gauge, we are free to set $u\equiv 1$ and think of $g=e^\alpha$ as a conformal change on $L$. The curvature of $A$ and the covariant derivative $\nabla_A\Phi$ are transformed accordingly under $g$:
\begin{align*}
F_A&\mapsto F_A+id*d\alpha,\\
\nabla_A\Phi&\mapsto e^{\alpha}(\nabla_A\Phi+2(d\alpha)^{1,0}\otimes\Phi).
\end{align*}

For $(A,\Phi)=g(A_0,\Phi_0)$, the second of (\ref{vortex3}) is satisfied. Let us define the moment map by the formula
\begin{align*}
\mu: L^2_2(\Sigma, \R)&\to L^2(\Sigma,\R)\\
\alpha &\mapsto *iF_A+\half(|\Phi|^2-1)=\Delta\alpha+\half |\Phi_0|^2(e^{2\alpha}-1)+h,
\end{align*}
where $h=*iF_{A_0}+\half(|\Phi_0|^2-1)\in C^\infty(\Sigma)$ is a smooth background function. It suffices to find $\alpha$ so that 
\begin{equation}\label{moment}
\mu(\alpha)=0. 
\end{equation}

First, $\mu$ is well-defined. By Sobolev embedding theorem, $L^2_2\embed L^\infty$ in dimension 2 and hence $e^\alpha-1\in L^\infty(\Sigma)\embed L^2$. Secondly, the solution to equation (\ref{moment}), if exist, must be unique. Suppose we have $\mu(\alpha_1)=\mu(\alpha_2)=0$, take $\gamma=\alpha_2-\alpha_1$. Then we have 
\[
\Delta\gamma+\half |\Phi_0|^2e^{2\alpha_1}(e^{2\gamma}-1)=\mu(\alpha_2)-\mu(\alpha_1)=0.
\] 

This implies:
\begin{align*}
0&= \int_{\Sigma}\langle \gamma, \Delta\gamma+\half |\Phi_0|^2e^{2\alpha_1}(e^{2\gamma}-1)\rangle \\
&= \int_\Sigma |\nabla \gamma|^2+ \half\int_\Sigma |\Phi_0|^2e^{2\alpha_1}(e^{2\gamma}-1)\gamma.
\end{align*}

Terms in the second line are non-negative. This shows $\nabla\gamma\equiv 0$ and $\gamma$ is a constant function on $\Sigma$. Since $x(e^{2x}-1)=0$ iff $x=0$, $\gamma\equiv 0$. 

To establish the existence of the solution, we apply variational principle. We define the energy functional:
\begin{align}
\E:L^2_2(\Sigma,\R)\to \R,\ \alpha\mapsto \half\int_\Sigma |\mu(\alpha)|^2.
\end{align}

This functional is well-defined on $L^2_2$, but we will not use this space as the variational space. For a solution to (\ref{moment}) to exist, we necessarily have 
\[
0< \half\int_\Sigma |\Phi_0|^2e^{2\alpha}=\half\int_{\Sigma}|\Phi|^2=\int_\Sigma (\half -iF_A)=\half Vol(\Sigma)- 2\pi c_1(L)\colonequals c.
\]

This explains the reason why equation (\ref{vortex3}) is subject to the solvability condition $c_1(L)<\frac{1}{4\pi} Vol(\Sigma)$. From now on, let's fix this positive number $c>0$ and associate to $c$ a subset of $L^2_2(\Sigma)$:
\begin{equation}\label{HC}
H_c=\{\alpha\in L^2_2(\Sigma):  \half\int_\Sigma |\Phi_0|^2e^{2\alpha}=c\}.
\end{equation}

Equivalently, an element $\alpha$ lies in $H_c$ if and only if $\alpha\in L_2^2$ and $\int_\Sigma \mu(\alpha)=0$. We will look for a minimizer of $\E(\alpha)$ for $\alpha\in H_c$.

Consider the decomposition $L^2_2(\Sigma)=H^\perp\oplus \R$, where $H^\perp$ is the subspace of $L_2^2$ that is $L^2$-orthogonal to constant functions. $H_c$ can be viewed as a graph over $H^\perp$; indeed, for $\alpha=\alpha_0+\alpha_1$ with $\alpha_0\in \R$ and $\alpha_1\in H^\perp$, $\alpha\in H_c$ if and only if
\begin{align}\label{solvability}
e^{2\alpha_0}=( \half\int_\Sigma |\Phi_0|^2e^{2\alpha_1})^{-1}\cdot c.
\end{align}

The crucial step to finding a minimizer of $\E$ is an a priori estimate:
\begin{theorem}\label{apriori1}
	There is a function $\eta:\R\to \R$ such that for any $C>0$ and $\alpha\in H_c$ with $\E(\alpha)<C$, $\|\alpha\|_{L_2^2}<\eta(C)$. 
\end{theorem}
\begin{proof}
	This is a consequence of the energy equation. The Bogomol'nyi transformation allows us to write for any configuration $(A,\Phi)$:
	\begin{align*}
	\int_\Sigma 2|\bpartial_A\Phi|^2+ |*iF_A+\half(|\Phi|^2-1)|^2=\E_{an}-\E_{top},
	\end{align*}
	where 
	\begin{align*}
	\E_{an}&= \int_\Sigma |F_A|^2+|\nabla_A\Phi|^2+\frac{1}{4}(1-|\Phi|^2)^2,\ \E_{top}=\int_\Sigma iF_A=2\pi c_1(L). 
	\end{align*}
	
Let $(A,\Phi)=e^\alpha\cdot (A_0,\Phi_0)$. It follows that
	\begin{align}\label{energy equations}
	2\E(\alpha)&= \int_\Sigma |\Delta\alpha+*iF_{A_0}|^2+ \text{positive terms}-\E_{top}.
	\end{align}
	
	Since $*iF_{A_0}$ is a smooth function on $\Sigma$,  (\ref{energy equations}) implies
	\[
	\|\Delta\alpha\|_2^2< aC+b
	\]
	for some $a,b>0$. Suppose $\lambda_1$ is the first positive eigenvalue of $\Delta$. Since $\Delta\alpha=\Delta\alpha_1$ and $\alpha_1$ is orthogonal to $\ker \Delta$, it follows that 
	\[
	\|\alpha_1\|_{L^2}\leq \frac{1}{\lambda_1}\|\Delta\alpha\|_2.
	\]
	Now we know $\|\alpha_1\|_{L^2_2}$ is controlled by $\eta_1(C)$ for some function $\eta_1$. By the solvability constraint (\ref{solvability}), $\alpha_0$ is determined by $\alpha_1$ and so $\|\alpha\|_{L^2_2}<\eta(C)$ for some $\eta$.
\end{proof}
\begin{proof}[Proof of Theorem \ref{B1}]
Let $a=\inf_{\alpha\in H_c} \E(\alpha)$. We can find a sequence of elements ${\alpha_n}\in H_c$ such that $\E(\alpha_n)\to a$ as $n\to\infty$. By Theorem \ref{apriori1}, $L^2_2$-norms of $\alpha_n$ are uniformly bounded, so we can find a converging subsequence in weak $L_2^2$-topology. Let us assume it is just the sequence itself and let $\alpha_\infty\in L^2_2$ be their limit. Since $L_2^2\embed L^\infty$ is compact, $\mu(\alpha_n)\to \mu(\alpha_\infty)$ weakly in $L^2$ and $\alpha_\infty\in H_c$. This shows $\E(\alpha_\infty)\leq \liminf \E(\alpha_n)=a$, so $\E(\alpha_\infty)=a$. Now Theorem \ref{B1} will follow from a lemma.
\end{proof}
\begin{lemma}\label{B.3}
If $\alpha\in H_c$ is a critical point of $\E|_{H_c}$, then $\mu(\alpha)=0$.
\end{lemma}
\begin{proof}[Proof of Lemma]
Let $f=\mu(\alpha)$. Any $\gamma\in L^2_2$ subject to the constraint:
\begin{equation}\label{linearconstr}
\int_\Sigma |\Phi_0|^2e^{2\alpha}\gamma=0,
\end{equation}
lies in the tangent space $T_\alpha H_c$. Since $\alpha$ is a critical point, it follows that
\begin{equation}\label{critialpoint}
0=\frac{d}{dt}|_{t=0}\E(\alpha+t\gamma)=\langle f, \mathcal{D}_\alpha \mu(\gamma)\rangle,
\end{equation}
where $\mathcal{D}_\alpha\mu(\gamma)=\Delta\gamma+e^{2\alpha}|\Phi_0|^2\gamma$ is the linearized operator at $\alpha$. 
\begin{lemma}\label{selfadjoint}
	The linearized operator $\mathcal{D}_\alpha\mu$ is self-adjoint on $L^2_2(\Sigma)$. 
\end{lemma}
\begin{proof}
	It is clear that $\mathcal{D}_\alpha\mu$ is well-defined on $L^2_2$ and it is symmetric. To show it is self-adjoint, it suffices to prove show $\gamma\in L^2$ and $\mathcal{D}_\alpha\mu(\gamma)\in L^2$ imply $\gamma\in L^2_2$. But this is trivial: $e^{2\alpha}|\Phi_0|^2\in L^\infty$ implies $e^{2\alpha}|\Phi_0|^2\gamma\in L^2$, so $\Delta\gamma\in L^2$. 
\end{proof}

Back to the proof of Lemma \ref{B.3}. Now Lemma \ref{selfadjoint} and relation (\ref{critialpoint}) imply $f\in L^2_2$ since $f$ is in the domain of the adjoint operator $(\mathcal{D}_\alpha\mu)^*$. Let $f=f_0+f_1$ with $f_0$ constant and $f_1$ subject to constraint (\ref{linearconstr}), i.e.
\[
\int_\Sigma |\Phi_0|^2e^{2\alpha}f_1=0. 
\]

Take $\gamma=f_1$ in (\ref{critialpoint}) and integrate by parts shows
\[
0=\|df_1\|_2^2+\int_\Sigma e^{2\alpha}|\Phi_0|^2|f_1|^2.
\]

Hence, $f_1\equiv 0$ and $f\equiv f_0$ is a constant function on $\Sigma$.  On the other hand, the constraint (\ref{HC}) shows 
\begin{align}
\Vol(\Sigma)\cdot f_0&= \int_\Sigma f= \int_\Sigma \mu(\alpha)= 0,
\end{align}
and hence $f=f_0=0$.
\end{proof}
\bigskip
Let us end this appendix by pointing out what is modified if $\Sigma$ is replaced by $\C$:
\begin{itemize}
	\item There is no solvability constraint for the vortex equation on $\C$. It is easier to show for a critical point $\alpha$ of $\E$, $\mu(\alpha)$ has to be zero. 
	\item Choosing a smaller variational space $(\ref{HC})$ is necessary for the proof of Theorem \ref{apriori1}. In fact, if we worked with $L_2^2$, Theorem \ref{apriori1} would be false since $\alpha_0$ can be arbitrarily negative while $\mu(\alpha)$ remains bounded. However, when it is $\C$, $L_2^2$ is the right space to work with. 
	\item The spectrum of Laplacian operator on $\C$ is continuous. In the proof of Theorem $\ref{apriori1}$, we have used discreteness of the spectrum in an essential way; we used the decomposition $L_2^2=H^\perp\oplus \R$. On $\C$, we will apply a cut-off function on the frequency space and decompose $\alpha$ into high-frequency and low-frequency parts.  
	\item We will establish Theorem \ref{apriori1} for $Y=\C$ (Theorem \ref{apriori2}) and $X=\C\times\Sigma$ (Theorem \ref{apriori3}), but their proofs will be much harder. It is the main technical issue when we apply variational principle. 
\end{itemize}

\section{The vortex equation on $\C$}\label{C}

In this appendix, we prove the existence of vortices on $\C$. Theorem \ref{C1} was originally proved in \cite{Taubes}. The proof presented below is based on Appendix \ref{B} and Lemma \ref{ABlemma}. 

To start, let $L_0$ be the trivial line bundle on $\C$ with the product metric. Then the polynomial
\[
\Phi_0= \prod_{i=1}^m (z-z_i)^{n_i},
\] 
is a holomorphic section that vanishes at $z_1,\cdots, z_m\in\C$ with multiplicity $n_1,\cdots, n_m$.

The vortex equation on $\C$ is defined by same formula (\ref{vortex3}) for the trivial line bundle $L_0$. 
\begin{theorem}[\cite{Taubes}]\label{C1}
For any effective divisor $D=\sum n_iz_i$, there is a unique solution $(A,\Phi)$ to the equation $(\ref{vortex3})$ up to gauge such that $Z(\Phi)=D$. 
\end{theorem}
\begin{proof}
Note that $(A_0=d, \Phi_0)$ is not the solution that we look for: $\Phi_0\notin L^\infty(\C)$. We choose a background conformal change. Set $\alpha_0=-\sum_{i=1}^m \frac{n_i}{2}\log (1+|z-z_i|^2)$. Then we obtain
\[
(A_1,\Phi_1)=e^{\alpha_0}\cdot (A_0,\Phi_0)= (d+*d\alpha_0, \prod_{i=1}^m \frac{(z-z_i)^{n_i}}{(1+|z-z_i|^2)^{n_i/2}}).
\]
For $(A_1,\Phi_1)$, the second equation of (\ref{vortex3}) is satisfied automatically. We wish to find a further conformal transformation $\alpha$  so that the first equation is satisfied for $e^{\alpha}\cdot (A_1,\Phi_1)$. This is equivalent to finding $\alpha\in L_2^2$ so that $\mu(\alpha)+h=0$ where 
\begin{align*}
\mu: L^2_2(\C)&\to L^2(\C),\\
\alpha&\mapsto \Delta\alpha+\half |\Phi_1|^2(e^{2\alpha}-1)
\end{align*}
is the moment map and the term $h\colonequals \Delta \alpha_0+\half(|\Phi_1|^2-1)$
comes from the background configuration $(A_1,\Phi_1)$. By Trudinger's inequality (Theorem \ref{A2}), $\mu$ is well-defined and by direct computation, $h\in L^2(\C)$. Our goal is to show $\mu$ gives a bijection between $L^2_2(\C)$ and $L^2(\C)$. In particular, there is a unique $\alpha\in L^2_2$ such that $\mu(\alpha)=-h$.

We start with the easy part of the proof. For any $g\in L^2(\C)$, define the associated energy functional $\E_g$ by the formula 
\[
\E_g(\alpha)= \half \int_\C |\mu(\alpha)-g|^2
\]
which measures the $L^2$ distance between $\mu(\alpha)$ and $g$. 
\begin{lemma}\label{minimum}
	For any critical point $\alpha$ of $\E_g$, we must have $\E_g(\alpha)=0$.
\end{lemma}
\begin{proof}
	Let $f= \mu(\alpha)-g$. Since $\alpha$ is a critical point, for any $\gamma\in L^2_2(\C)$, 
	\begin{equation}\label{criticalpoint2}
	0=\frac{d}{dt}|_{t=0} \E(\alpha+t\gamma)=\langle f, \mathcal{D}_\alpha \mu(\gamma)\rangle.
	\end{equation}
	Since $\mathcal{D}_\alpha\mu(\gamma)=\Delta\gamma+  |\Phi_1|^2e^{2\alpha}\gamma$ is self-adjoint on $L^2_2$, (\ref{criticalpoint2}) implies $f\in L^2_2$. Then we plug $\gamma=f$ into (\ref{criticalpoint2}). Integration by parts shows
	\[
	0= \|df\|_2^2+\int_\C |\Phi_1|^2e^{2\alpha}|f|^2,
	\] 
	and $f$ has to be zero everywhere. 
\end{proof}

In light of Lemma \ref{minimum}, it suffices to find a point $\alpha$ that realizes the infimum of $\E_g$. To do this, we construct a minimizing sequence $\{\alpha_n\}\subset L^2_2$ such that $\E_g(\alpha_n)\to \inf \E_g$. The hardest part of the proof is an a priori estimate, the counterpart of Theorem \ref{apriori1} and Theorem \ref{apriori3}:
\begin{theorem}\label{apriori2}
	There is a function $\eta:\R\to \R$ such that for any $C>0$ and $\alpha\in L^2_2(\C)$ with $\E_g(\alpha)<C$, we have $\|\alpha\|_{L^2_2}<\eta(C)$. 
\end{theorem}
\begin{proof}
	Note that it suffices to prove this theorem for one special $g$ and the rest will follow by triangle inequality. We choose $g=-h$ and write $\E=\E_{-h}$ for short. By Bogomol'nyi transformation, we have energy equation:
	\begin{equation}\label{energy1}
	2\E(\alpha)=-2\pi d+\int_\C |\Delta\alpha+*iF_{A_1}|^2+ |\nabla_A\Phi|^2+\frac{1}{4}(1-e^{2\alpha}|\Phi_1|^2)^2,
	\end{equation}
	where $d=\sum_{i=1}^m n_i$ is the degree of $\Phi_0$ and $(A,\Phi)=e^\alpha\cdot (A_1,\Phi_1)$. Since $1-|\Phi_1|^2, *iF_{A_1}\in L^2(\C)$, we know from (\ref{energy1}) that
	\[
	\int_\C |\Delta\alpha|^2,\ \int_\C (1-e^{2\alpha})^2|\Phi_1|^4 < aC+b\]
	for some $a,b>0$. It suffices to control $\|\alpha\|_2$. We decompose $\C$ into two parts
	\[
	\C=A_1\coprod A_2, A_1=\{z\in\C: \alpha(z)>-1, |\Phi_1|^2>\epsilon\},\ A_2=A_1^c.
	\] 
	
	Then 
	\[
	\int_{A_1} |\alpha|^2\leq \frac{C_1}{\epsilon^2} \int_\C|(1-e^{2\alpha})|\Phi_1|^2|^2.
	\]
	Since $Z_\epsilon(\Phi_1)\colonequals\{|\Phi_1|^2\leq\epsilon\}\subset\C$ is compact, we take $R\gg 0$ such that $Z_\epsilon(\Phi_1)\subset B(0,R)$. Then
	\[
	Area(A_2\backslash B(0,R))\leq \frac{1}{\epsilon^2(1-e^{-2})^2} \int_\C|(1-e^{2\alpha})|\Phi_1|^2|^2.
	\]
	
	Now we are in the place to apply Lemma \ref{ABlemma}. 
\end{proof}

Theorem \ref{apriori2} allows us to find a weakly convergent subsequence among $\{a_n\}$. Denote this limit by $\alpha_\infty$. We know $\E_g(\alpha_\infty)\leq \lim\E_g(\alpha_n)=\inf \E_g$ and hence $\alpha_\infty$ is a critical point of $\E_g$. Now we use Lemma \ref{minimum} to conclude.
\end{proof}
\bigskip

The proof of Theorem \ref{apriori3} is modeled on the proof above. It is much harder to work with $X=\C\times \Sigma$ due to some technical reasons:
\begin{enumerate}
	\item The $L^2$-norm of $\Delta_\Sigma\alpha$ is not a term in the energy equation. We worked very hard in Lemma \ref{fiber} to show it is actually controlled by the analytic energy.
	\item In dimension $4$, the thickened zero locus $Z_\epsilon(\sigma_1)=\{|\sigma_1|^2<\epsilon\}$ is no longer a compact region. This is the reason why there are two classes of points in the good set $A_1$ in the proof of Theorem \ref{apriori3}. For the first class, $\alpha$ has large variation on the fiber. For the second, its variation is small and hence $\alpha$ does not ``see" the zero locus of $\sigma_1$ on that fiber. 
\end{enumerate} 

\bibliographystyle{alpha}
\bibliography{sample}

\end{document}